\documentclass[11pt]{amsart}
\usepackage{amsmath,amsthm,amsfonts,latexsym,amssymb}
\usepackage{graphicx}
\usepackage[arrow, matrix, curve]{xy}
\usepackage{float}

\calclayout\evensidemargin .1in

\newcommand{\R}{\mathbb{R}}

\newtheorem{theorem}{Theorem}[section]
\newtheorem{lemma}[theorem]{Lemma}
\newtheorem{proposition}[theorem]{Proposition}

\theoremstyle{definition}

\newtheorem{example}[theorem]{Example}
\newtheorem{remark}[theorem]{Remark}

\newtheorem{definition}[theorem]{Definition}

\title[Compatible Relative Open Books on Relative Contact Pairs]{Compatible Relative Open Books on Relative Contact Pairs \\ via Generalized Square Bridge Diagrams}

\author{\. I. \"Ozge Ta\c{s}p\i nar}
\address{Dept. of Mathematics, Middle East Technical University, Ankara, TURKEY}
\email{isarac@metu.edu.tr}

\author{M. Firat Arikan}
\address{Dept. of Mathematics, Middle East Technical University, Ankara, TURKEY}
\email{farikan@metu.edu.tr}

\date{\today}

\begin{document}
	
	\begin{abstract}
		Using square bridge position, Akbulut-Ozbagci and later Arikan gave algorithms both of which construct an explicit compatible open book decomposition on a closed contact $3$-manifold which results from a contact $(\pm 1)$-surgery on a Legendrian link in the standard contact $3$-sphere. In this article, we introduce the ``generalized square bridge position'' for a Legendrian link in the standard contact $5$-sphere and partially generalize this result to the dimension five via an algorithm which constructs relative open book decompositions on relative contact pairs.
	\end{abstract}
	
	\maketitle
	
	
	\section{Introduction}
	
	Let $Y$ be an oriented $(2n+1)$-dimensional manifold. A \textit{(co-oriented positive) contact structure} on $Y$ is a maximally non-integrable hyperplane field  $\xi=\textrm{Ker}(\alpha)$ where $\alpha$ is a smooth  $1$-form on $Y$ such that $\alpha \wedge(d\alpha)^n> 0$. The pair $(Y,\xi)$, or more simply $Y_\xi$, is called a\textit{ contact manifold} and $\alpha$ is called a \textit{contact form} for $\xi$. An \textit{open book (decomposition)} on $Y$ is a pair $(B, \pi)$ where $B$ (called \textit{the binding}) is a codimension two submanifold with a trivial normal bundle $B\times D^2$, and $\pi: Y\setminus B \rightarrow \mathbb{S}^1$ is a smooth fibration map which agrees with the angular coordinate on the $D^2-$factor near $B$. The closure of the fiber $F= \overline{\pi^{-1}(\theta)}$ is called the \textit{page} over $\theta \in \mathbb{S}^1$. The \textit{monodromy} of $(B, \pi)$ is the self-diffeomorphism $h:F\to F$ defined by the return map of the flow of the vector field $\partial/\partial \theta$. Up to \textit{contactomorphism} (a diffeomorphism respecting contact structures) an open book can be also given abstractly as a pair $(F,h)$. An open book $(B, \pi)$ and $\xi$ on $Y$ are \textit{compatible} if there exists a contact form $\alpha$ for $\xi$ such that $\alpha$ is a contact form on $B$, $d\alpha$ defines a symplectic structure on each page, and the orientation on $B$ induced from $\alpha$ agrees with the one induced from $d\alpha$ considering $B$ as the contact boundary of a symplectic page. For more details see \cite{Gd} and \cite{Et1}.\\
	
	A $k$-dimensional ($k\leq n$) submanifold $\Sigma$ of $Y_\xi$ is \textit{isotropic} if $T_p\Sigma\subset \xi_p$ for all $p \in \Sigma$. An $n$-dimensional isotropic submanifold is said to be \textit{Legendrian}. Any Legendrian $\Sigma\subset Y_\xi$ comes with a \textit{contact} (or \textit{Thurston-Bennequin}) \textit{framing} given by a vector field along $\Sigma$ which is transverse to $\xi$. For a Legendrian knot $K$ in $(\mathbb{R}^3,\xi_0^3=\textrm{Ker}(\alpha_0^3=dz+x_1dy_1))$, its contact framing can be computed as the integer
	$$ tb(K)=writhe(K)-\frac{1}{2} \# (\text{cusps of } K), $$ 
	called the \textit{Thurston-Bennequin number} of $K$, where $writhe(K)$ is the sum of the signs of the crossings in the \textit{front projection} of $K$ onto the $y_1z-$plane. In higher dimensional standard contact spaces, the Thurston-Bennequin number of a Legendrian submanifold coincides with a topological invariant and one can easily be computed (see \cite{EES}).\\
	
	A \textit{contact surgery} along a Legendrian $n$-\textit{knot} $K\subset Y_\xi$ (i.e., the image of an $n$-sphere under a Legendrian embedding) is an operation which results in a new \textit{surgered} contact manifold. The surgered contact manifold is formed by first removing a tubular neighborhood $N$ of $K$ in $Y_\xi$, and then gluing $S^n\times D^{n+1}$ back using a chosen diffeomorphism from $S^n \times S^n$ onto  $\partial N$ so that contact distributions of the pieces smootly glued together. The \textit{contact surgery coeficient} (which determines the gluing diffeomorphism) is measured relative to the contact framing of $K$. (See  Section \ref{sec:admissible-surgery} for details). In dimension three, Ding and Geiges proved the following result:
	
	\begin{theorem} [\cite{DG}] 
		Every closed contact $3$-manifold can be obtained by a contact $(\pm 1)$-surgery on a Legendrian link in the standard contact $3$-sphere $\mathbb{S}^3_{st}=(\mathbb{S}^3,\xi_{st}^3)$.
		\label{thm:suregery in 3}
	\end{theorem}
	
	\indent This result implies that any closed contact $3$-manifold can be described by a contact surgery diagram where contact surgery coefficients are $\pm 1$'s. Also it was shown (see below) how contact $(\pm 1)$-surgery effects a compatible open book when a surgery knot lies on a page:
	
	\begin{theorem} [\cite{Et1}] 
		Let $(F,h)$ be an open book compatible with the contact manifold $(M^3,\eta)$. If a Legendrian knot $K$ lies on a page $F$, then  the surgered manifold $(M^{\pm 1},\eta^{\pm 1})$ obtained by the contact $(\pm 1)$-surgery along $K$ has a compatible open book $(F, h\circ {\tau_{K}}^{\mp})$ where $\tau_{K}^{\mp 1}$ denotes the negative/positive (or left-handed/right-handed) Dehn twist along $K$, respectively.
		\label{thm:Etnyresurgery}
	\end{theorem}
	
	Akbulut and Ozbagci gave an algorithm constructing an open book on $\mathbb{S}^3_{st}$ so that a given Legendrian link lies on a page which is the Seifert surface of a suitable $(p,q)$-torus link \cite{AO}. Arikan improved this algorithm by using contact cell decompositions \cite{Ar}. Combining these results with Theorem \ref{thm:Etnyresurgery}, one can easily construct an open book on any closed contact $3$-manifold obtained by a contact $(\pm 1)$-surgery along a Legendrian link in $\mathbb{S}^3_{st}$. \\
	
	In higher dimensions, some similar results along the same lines are as follows: Koert constructed an open book on the standard contact $5$-sphere $\mathbb{S}^5_{st}=(\mathbb{S}^5,\xi_{st}^5)$ such that its pages are Briskorn varieties $V_{\epsilon}(p,q,2)$ and its binding is Seifert filling of a Briskorn manifold $\Sigma (p,q,2) $ \cite{Koert}. The right-handed (generalized) Dehn twist $\tau: T^{*}S^n\rightarrow T^{*}S^n$ along the zero section of the cotagent bundle of $n$-spheres and some of its applications were studied by Seidel \cite{Seidel}. A definition of contact $(\pm1)$-surgery on a Legendrian sphere based on a generalized Seidel-Dehn twist was given by Avdek \cite{Av}. Lazarev proved that if $(Y^{2n+1},\xi)$, $n \geqslant 2$ with an almost Weinstein filling,  then it can be obtained from $\mathbb{S}^{2n+1}_{st}$ by a sequence of isotropic surgeries (of any index at most $n$)  and a single coisotropic surgery (of index $n$) \cite{L }. Conway and Etnyre generalized Theorem \ref{thm:suregery in 3} to any odd dimension:
	
	\begin{theorem} [\cite{CE}] 
		Any oriented contact  manifold $(Y^{2n+1},\xi)$ can be obtain from $\mathbb{S}^{2n+1}_{st}$ by a sequence of  isotropic and coisotropic surgeries of all indicies.
		\label{thm:generalized}
	\end{theorem}
	
	In this paper, we focus on a certain class of closed contact $5$-manifolds obtained by a contact $(\pm1)$-surgeries along a special class of Legendrian $2$-\textit{links} (i.e., links of $2$-spheres) inside $\mathbb{S}^5_{st}$ as we describe below. We introduce a new diagram for such a $2$-link which can be considered as a generalization of square bridge diagram of a Legendrian $1$-link in $\mathbb{R}^3$ defined by Lyon \cite{Ly}. After presenting the \textit{generalized square bridge diagram} in Section \ref{sec:generalized-square-bridge}, we will prove our main result on the existence of open books on a certain class of closed contact $5$-manifolds in Section \ref{sec:algorithm}. We will give examples in Section \ref{sec:examples} to demonstrate how our algorithm works. In order to state the main theorem, one needs to introduce some relative notions first. We refer the reader to Section \ref{sec:sub_free_contact} for more details.\\
	
	Consider the standard contact structures $\xi_0^5=\textrm{Ker}(\alpha_0^5=dz+x_1dy_1+x_2dy_2)$ on $\mathbb{R}^5$ with coordinates $(x_1,y_1,z,x_2,y_2)$, and $\xi_0^3=\textrm{Ker}(\alpha_0^3=dz+x_1dy_1)$ on $\mathbb{R}^3$ with coordinates $(x_1,y_1,z)$ as above. Note that $(\mathbb{R}^3,\xi_0^3)$ (resp., $\mathbb{S}^3_{st}$) is a contact submanifold of $(\mathbb{R}^5,\xi_0^5)$ (resp., $\mathbb{S}^5_{st}$). Also recall that $$(\mathbb{R}^3,\xi_0^3)\cong\mathbb{S}^3_{st}\setminus \{pt\}, \quad (\mathbb{R}^5,\xi_0^5)\cong\mathbb{S}^5_{st}\setminus \{pt\}.$$ Therefore, any isotropic submanifold of $\mathbb{S}^3_{st}$ (resp., $\mathbb{S}^5_{st}$) is also an isotropic submanifold of $(\mathbb{R}^3,\xi_0^3)$ (resp., $(\mathbb{R}^5,\xi_0^5)$), and vice versa. This will enable us to work inside $(\mathbb{R}^3,\xi_0^3)$ and $(\mathbb{R}^5,\xi_0^5)$. 
	
	\begin{definition}
		Let $\mathbb{L}=\bigsqcup_{i=1}^{r} L_{i} \subset (\mathbb{R}^5,\xi_0^5)\subset \mathbb{S}^5_{st}$ be a Legendrian $2$-link. $\mathbb{L}$ is called a \textit{relative link} if its equatorial link $\mathbb{K}=\mathbb{L}\cap \mathbb{S}^3_{st}=\bigsqcup_{i=1}^{r} {K}_{i} $ is a Legendrian $1$-link in $(\mathbb{R}^3,\xi_0^2)\subset \mathbb{S}^3_{st}$. A relative Legendrian $2$-link $\mathbb{L}$ is called \textit{admissible} if the following two conditions are satistified:
		\begin{itemize}
			\item[(i)] Each link component $L_{i}$ is Legendrian unknotted $2$-sphere in $\mathbb{S}^5_{st}$.
			\item[(ii)] The Legendrian front of each link component ${K}_{i}={L}_i\cap \mathbb{S}^3$ onto the $y_1z$-plane has no self-crossing. 
		\end{itemize} 
		A(n) \textit{(admissible) relative $2$-knot} is a(n) (admissible) relative $2$-link with a single component.
	\end{definition}
	
	\begin{definition} [See also Section \ref{sec:admissible-surgery}] Let $\mathbb{L}\subset\mathbb{S}^5_{st}$ be a(n) (admissible) relative Legendrian $2$-link with its equatorial link $\mathbb{K}\subset\mathbb{S}^3_{st}$. Suppose the (closed) contact $5$-manifold $Y^5_{\xi}$ is obtained by only contact $(\pm 1)$-surgeries along $\mathbb{L}$, and the corresponding contact $(\pm 1)$-surgeries along $\mathbb{K}$ in $\mathbb{S}^3_{st}$ produce the (closed) contact submanifold $M^3_{\eta}\subset Y^5_{\xi}$. Such a surgery will be called a(n) \textit{(admissible) relative contact} $(\pm 1)$-\textit{surgery} along a pair $(\mathbb{L},\mathbb{K})$. The pair $(Y^5_{\xi}, M^3_{\eta})$ will be called a(n) \textit{(admissible) relative contact pair}. 
	\end{definition}
	
	We note that by forgetting all the contact geometric terms in the above definitions, one can similarly introduce a \textit{relative (smooth) topological pair} $(Y^5, M^3)$ obtained by a relative topological surgery along a pair $(\mathbb{L},\mathbb{K})\subset(\mathbb{S}^5, \mathbb{S}^3)$ where $\mathbb{K}=\mathbb{L}\cap \mathbb{S}^3$ is the equatorial link.\\
	
	A \textit{relative open book (decomposition)} on a relative (smooth) topological pair $(Y^5, M^3)$ is an open book $\mathcal{OB}_Y=(B, \pi)$ of $Y$ such that $\mathcal{OB}_M=(B\cap M, \pi|_{M\setminus (B\cap M)})$ defines an open book of $M$. For a relative contact pair $(Y^5_{\xi}, M^3_{\eta})$, we say that $(\xi, \eta)$ is \textit{compatible} with the relative open book $(\mathcal{OB}_Y,\mathcal{OB}_M)$ on $(Y^5, M^3)$ if the contact structures are compatible with the corresponding open books. Now we are ready to state our main theorems generalizing the results in \cite{Ar}:
	
	\begin{theorem} \label{thm:main thm}
		For any admissable Legendrian $2$-link $\mathbb{L}$ in $\mathbb{S}^5_{st}$ with its equatorial link $\mathbb{K}\subset\mathbb{S}^3_{st}$, there exists a compatible relative open book $(\mathcal{OB}_{\mathbb{S}^5},\mathcal{OB}_{\mathbb{S}^3})$ on the (trivial) relative contact pair $(\mathbb{S}^5_{st}, \mathbb{S}^3_{st})$ such that for any component $L$ of $\,\mathbb{L}$ and its equatorial knot $K\subset \mathbb{K}$, we have
		
		\begin{enumerate}
			\item[(i)] $K$ lies on a Weinstein page $F_{\mathbb{S}^3}$ of the (compatible) open book $\mathcal{OB}_{\mathbb{S}^3}$ of $\mathbb{S}^3_{st}$,
			\item[(ii)] the page framing of $K$ is equal to its contact framing in $\mathbb{S}^3_{st}$,
			\item[(iii)] $L$ lies on a Weinsten page $F_{\mathbb{S}^5}$ of the (compatible) open book $\mathcal{OB}_{\mathbb{S}^5}$ of $\mathbb{S}^5_{st}$,
			\item[(iv)] the page framing of $L$ is equal to its contact framing in $\mathbb{S}^5_{st}$, and
			\item[(v)] $F_{\mathbb{S}^3}$ is symplectically and properly embedded in $F_{\mathbb{S}^5}$, and the monodromies of $\mathcal{OB}_{\mathbb{S}^5},\mathcal{OB}_{\mathbb{S}^3}$ are, respectively, the products of right-handed Dehn twists $$h_{\mathbb{S}^5}=\tau_{\mathcal{D}_1}\circ \tau_{\mathcal{D}_2}\circ\cdots \circ \tau_{\mathcal{D}_m}, \quad h_{\mathbb{S}^3}=\tau_{\gamma_1}\circ \tau_{\gamma_2}\circ \cdots \circ \tau_{\gamma_m}$$ (for some $m$) where each $\gamma_k$ is an isotropic equatorial circle of  a Legendrian sphere $\mathcal{D}_k$.
		\end{enumerate}
	\end{theorem}
	
	\begin{theorem} \label{thm:contact surgery}
		Let $(Y^5_{\xi}, M^3_{\eta})$ be any admissible relative contact pair obtained by an admissible relative contact $(\pm 1)$-surgery along an admissible link $\mathbb{L}$ in $\mathbb{S}^5_{st}$. Then there exists a compatible relative open book $(\mathcal{OB}_Y,\mathcal{OB}_M)$ on $(Y^5_{\xi}, M^3_{\eta})$ with Weinstein pages. Moreover, the pages and the monodromies of the open books $\mathcal{OB}_Y$, $\mathcal{OB}_M$ are precisely and systematically determined by the contact surgery along $\mathbb{L}$.
	\end{theorem}
	
	
	
	\section{Stabilizations of Relative Open Books} \label{sec:sub_free_contact}  
	It is known that any odd dimensional manifold has an open book decomposition. Stabilization is a process which transforms a given open book of a manifold into another open book for the same manifold. It changes both pages and the monodromy. One can naturally generalize this process for the relative open books on relative contact pairs as what we discuss below. To this end, let us introduce relative open books in the most generality. As before we first define them for (smooth) topological pairs and then make a connection with contact geometry. Although any odd dimensions can be considered, we will restrict our attention to the pairs of dimension three and five only.
	
	\begin{definition} \label{def:Rel_Open_Book}
		Let $M^3$ and $Y^5$ be closed (smooth) topological manifolds such that $M$ is a submanifold of  $Y$. A \textit{relative open book (decomposition)} on the pair $(Y,M)$ is a quadruple $(B_Y, B_M, \pi_Y, \pi_M)$ such that 
		\begin{itemize}
			\item[(i)] $(B_M, \pi_M)$ and $(B_Y, \pi_Y)$ are open book decompositions of $M$ and $Y$, resp., and
			\item[(ii)] $\pi_M=\pi_Y \circ i_M$ where $i_M:M \hookrightarrow Y$ is the embedding map.
		\end{itemize}	
	\end{definition}
	
	Note that the conditions imply that  \;$\overline{\pi^{-1}_M(t)} \subset \overline{\pi^{-1}_Y(t)}$\; and  \;$\partial \overline{\pi^{-1}_M(t)} \subset \partial \overline{\pi^{-1}_Y(t)}$\; for all \;$t \in S^1$. This means that $2$-dimensional pages of $(B_M, \pi_M)$ are properly embedded in the corresponding $4$-dimensional pages of $(B_Y, \pi_Y)$. Moreover, the monodromy of $(B_Y, \pi_Y)$ restricts to the monodromy of $(B_M, \pi_M)$. These suggest that relative open books can be also defined in a purely abstract way as follows:
	
	\begin{definition} \label{def:Abstract_Rel_Open_Book}
		Let $M^3$ and $Y^5$ be closed (smooth) topological manifolds such that $M$ is a submanifold of  $Y$. An \textit{abstract relative open book (decomposition)} on the pair $(Y,M)$ is a quadruple $(F_Y,F_M,h_Y,h_M)$ such that
		\begin{itemize}
			\item[(i)] $(F_M, h_M)$ and $(F_Y, h_Y)$ are abstract open book decompositions of $M$ and $Y$, resp.,
			\item[(ii)]  the page $F_M$ is properly embedded in the page $F_Y$ (so, $\partial F_M\subset \partial F_Y$),
			\item[(iii)] $h_Y|_{F_M}=h_M$. 
		\end{itemize}
	\end{definition}
	
	In the presence of both geometric and topological structures, it is worth to try to relate them in some way. Giroux showed that there is one such relation, called \textit{compatibility} as we have mentioned earlier, between contact structures and open book decompositions, and proved that there is a one-to-one correspondence between contact structures (up to isotopy) and their compatible open books (up to positive stabilization) \cite{Gi}. One can easily modify compatibility notion for relative contact pairs and relative open book decompositions as follows:
	
	\begin{definition} \label{def:Compatibility_of_Rel_Open_Book}
		Let $(Y^5_{\xi}, M^3_{\eta})$ be a relative contact pair. A relative open book decomposition $(B_Y, B_M, \pi_Y, \pi_M)$ on $(Y,M)$ is \textit{compatible} with the pair $(\xi, \eta)$ if
		\begin{itemize}
			\item[(i)] the open book $(B_Y, \pi_Y)$ is compatible with the contact structure $\xi$ on $Y$, and
			\item[(ii)] the open book $(B_M, \pi_M)$ is compatible with the contact structure $\eta$ on $M$.
		\end{itemize}
	\end{definition}
	
	\begin{example} \label{ex:trivial_open_book}
		Consider the (trivial) relative contact pair $(\mathbb{S}^5_{st},\mathbb{S}^3_{st})$ of standard contact spheres as before. Let $\mathbb{S}^2$ be the standard Legendrian $2$-unknot in $\mathbb{S}^5_{st}$ and $\mathbb{S}^1$ be its isotropic equatorial circle which is the standard Legendrian $1$-unknot in $\mathbb{S}^3_{st}$ (see Figure \ref{Fig:fronts} for their front projections). Then it is a (customary) standard fact that $(D^*\mathbb{S}^2,\tau_{\mathbb{S}^2})$ and $(D^*\mathbb{S}^1,\tau_{\mathbb{S}^1})$ are compatible open books on $\mathbb{S}^5_{st}$ and $\mathbb{S}^3_{st}$, respectively, where $D^*\mathbb{S}^n$ denotes the unit disk cotangent bundle of the $n$-sphere. It is also well known that $D^*\mathbb{S}^1$ is properly embedded in $D^*\mathbb{S}^2$ and the right-handed Dehn twist $\tau_{\mathbb{S}^2}: D^*\mathbb{S}^2\rightarrow D^*\mathbb{S}^2$ restricts to the right-handed Dehn twist $\tau_{\mathbb{S}^1}: D^*\mathbb{S}^1\rightarrow D^*\mathbb{S}^1$. Therefore, $(D^*\mathbb{S}^2, D^*\mathbb{S}^1,\tau_{\mathbb{S}^2},\tau_{\mathbb{S}^1})$ is a compatible relative open book decomposition on the relative contact pair $(\mathbb{S}^5_{st},\mathbb{S}^3_{st})$.
	\end{example}
	
	According to the Giroux correspondence, the positive stabilization does not only fix the diffeomorphism type of the underlying smooth manifold, but it also respects to the contact structure. More precisely, positively stabilization of a compatible open book on a contact manifold is compatible with a contact structure (on the same smooth manifold) which is isotopic to the original one. This correspondence can be naturally generalized for compatible relative open books on contact relative pairs. To this end, one needs to define the notion of relative stabilization. For completeness, let us first recall stabilization which arises as a special case of a more general process, called plumbing:\\
	
	Let $(Y^{2n+1}_i, \xi_i)$ be two closed contact $(2n+1)$-manifolds such that each $\xi_i=\textrm{Ker}(\alpha_i)$ is compatible with an open book $(F_i,h_i)$ on $Y_i$. Suppose that $A_i$ is a properly embedded Lagrangian $n$-disk in a symplectic page $(F_i,d\alpha_i)$ with Legendrian boundary $\partial A_i \subset (\partial F_i, \textrm{Ker}(\alpha_i|_{\partial F_i}))$. By the Weinstein (i.e., Lagrangian) neighborhood theorem each $A_i$ has a standard neighborhood $N_i$ in $F_i$ which is symplectomorphic to $(T^* D^n, d\lambda_{\textrm{can}})$ where $\lambda_{\textrm{can}}=\textbf{p}\textbf{d}\textbf{q}$ is the canonical $1$-form on $\mathbb{R}^n \times \mathbb{R}^n$ with coordinates $(\textbf{q},\textbf{p})$. Then the \emph{plumbing} of the open books $(F_1,h_1)$ and $(F_2,h_2)$ \emph{along} $A_1$ and $A_2$ is the open book $(\mathcal{P}(F_1,F_2;A_1,A_2),h)$ on the connected sum $Y_1 \# Y_2$ with the pages obtained by gluing $F_i$'s together along $N_i$'s by interchanging $\textbf{q}$-coordinates in $N_1$ with $\textbf{p}$-coordinates in $N_2$, and vice versa. In order to describe $h$, extend each $h_i$ to $\tilde{h}_i$ on the new page by requiring $\tilde{h}_i$ to be identity map outside the domain of $h_i$. Then the monodrodmy $h$ is defined to be $\tilde{h}_1 \circ \tilde{h}_2$. For simplicity, ``tilde'' sign will be dropped, and we will simply write $h=h_1 \circ h_2$.
	
	\begin{definition}[\cite{Gi}] \label{def:Stabilization_Compatible_Open Book}
		Suppose that $(F,h)$ is an open book compatible with the contact structure $\xi=\textrm{Ker}(\alpha)$ on a closed $(2n+1)$-manifold $Y$. Let $A$ be a properly embedded Lagrangian $n$-disk in a symplectic page $(F,d\alpha)$ such that  its boundary $\partial A$ is a Legendrian $(n-1)$-sphere in the contact binding $(\partial F,\textrm{Ker}(\alpha |_{\,\partial F}))$.
		Then the \textit{positive stabilization }$\mathcal{S}^+[(F,h);A]$ of $(F,h)$
		\textit{along} $A$ is the open book defined by 
		$$\mathcal{S}^+[(F,h);A]\doteq(\mathcal{P}(F,D^*\mathbb{S}^n;A,\textbf{D}), h \circ \tau_{\mathbb{S}^n})$$ 
		where $\textbf{D}\cong D^n$ is any fiber in the unit disk cotangent bundle $D^*\mathbb{S}^n$ and $\tau_{\mathbb{S}^n}:D^*\mathbb{S}^n \to D^*\mathbb{S}^n$ is the right-handed Dehn twist along the zero section $\mathbb{S}^n$. Similarly, the \textit{negative stabilization }$\mathcal{S}^-[(F,h);A]$ of $(F,h)$ \textit{along} $A$ is the open book defined by 
		$$\mathcal{S}^-[(F,h);A]\doteq(\mathcal{P}(F,D^*\mathbb{S}^n;A,\textbf{D}), h \circ \tau^{-1}_{\mathbb{S}^n})$$
	\end{definition}
	
	\begin{remark} \label{rem:alternative_def_of_stabilization}
		When a compatible open book $(F,h)$ is stabilized, abstractly it means that each symplectic page $(F^{2n},d\alpha)$ gains a Weinstein handle $H=D^n \times D^n$ of index $n$. Attaching spheres of these handles are the copies of Legendrian $(n-1)$-sphere $\partial A$ considered inside the boundary $\partial F$ of each fiber. Also the notation used in the above definition is quite misleading and  is not appropriate for our purposes. When a positive/negative stabilization is performed, abstractly it means that the right-handed/left-handed Dehn twist that we adjoin to the monodromy is along the Lagrangian $n$-sphere $L$ formed by gluing the Lagrangian $n$-disk $(A,\partial A) \subset (F,\partial F)$ and the Lagrangian $n$-disk $D^2 \times \{0\}$ (the core disk of $H$) together along their common boundary $\partial A \subset \partial F$, that is, $L=A \cup_{\partial A} (D^2 \times \{0\})$. In conclusion, we will use the following (abstract) alternative (but equivalent) definition (notation) for positive/negative stabilization of $(F,h)$ \textit{along} $A$:
		$$\mathcal{S}^\pm[(F,h);A]\doteq(F\cup H, h \circ \tau^{\pm 1}_{L}).$$   
	\end{remark}
	
	\begin{theorem} [\cite{Gi}] \label{thm:Stabilization_Compatible_Open Book}
		In the setting of Definition \ref{def:Stabilization_Compatible_Open Book}, both $\mathcal{S}^\pm[(F,h);A]$ are open books on the connected sum $Y\# \mathbb{S}^{2n+1} (\cong Y)$. Indeed, $\mathcal{S}^+[(F,h);A]$ is an open book on the contact connected sum $Y_\xi \# \mathbb{S}^{2n+1}_{st} (\cong Y_\xi)$, in other words, it is a compatible open book on $(Y,\xi)$.
	\end{theorem} 
	
	Returning back to the dimension three and five, we now define relative stabilization process:
	
	\begin{definition} \label{def:Stab_of_Rel_Open_Book}
		Suppose that $(F_Y,F_M,h_Y,h_M)$ is a relative open book compatible with the pair of contact structures $(\xi=\textrm{Ker}(\alpha), \eta=\textrm{Ker}(\alpha|_{M}))$ on a relative contact pair $(Y^5_{\xi}, M^3_{\eta})$. Let $A$ be a properly embedded Lagrangian $2$-disk in a symplectic page $(F_Y,d\alpha|_{F_Y})$ with Legendrian boundary $\partial A$ (an unknot) in the contact binding $(\partial F_Y,\textrm{Ker}(\alpha |_{\,\partial F_Y}))$ such that $A\cap F_M$ is a properly embedded Lagrangian $1$-disk in a symplectic page $(F_M,d\alpha|_{F_M})$ with Legendrian boundary $\partial(A\cap F_M)$ (a pair of points) in the contact binding $(\partial F_M,\textrm{Ker}(\alpha |_{\,\partial F_M}))$. Also consider the relative Weinstein handle $H=(H^4,H^2)=(D^2\times D^2, D^1\times D^1)$ such that the core and the cocore of $H^2$ is properly embedded in those of $H^4$, that is, $D^1\times\{0\}\subset D^2\times\{0\}$, $\{0\}\times D^1\subset \{0\}\times D^2$ and $\partial D^1\times\{0\}\subset \partial D^2\times\{0\}$ and $\{0\}\times \partial D^1\subset \{0\}\times \partial D^2$. Then the \textit{positive relative stabilization} $\mathcal{S}^+[(F_Y,F_M,h_Y,h_M);A, A\cap F_M]$ of $(F_Y,F_M,h_Y,h_M)$ \textit{along} $(A,A\cap F_M)$ is the open book defined by 
		$$\mathcal{S}^+[(F_Y,F_M,h_Y,h_M);A, A\cap F_M]\doteq(F_Y\cup H^4,F_M\cup H^2, h_Y \circ \tau_L,h_M \circ \tau_K)$$ 
		where the framing of $H^4$ is one less than the contact framing of $\partial A$ in $(\partial F_Y,\textrm{Ker}(\alpha |_{\,\partial F_Y}))$, and $K=A \cap F_M \cup_{\partial (A \cap F_M)} (D^1 \times \{0\})$ is the Lagrangian $1$-sphere in the new symplectic page $F_M \cup H^2$ and $L=A \cup_{\partial A} (D^2 \times \{0\})$ is the Lagrangian $2$-sphere  in the new symplectic page $F_Y \cup H^4$. Similarly, the \textit{negative relative stabilization} $\mathcal{S}^-[(F_Y,F_M,h_Y,h_M);A, A\cap F_M]$ of $(F_Y,F_M,h_Y,h_M)$ \textit{along} $(A,A\cap F_M)$ is the open book defined by 
		$$\mathcal{S}^-[(F_Y,F_M,h_Y,h_M);A, A\cap F_M]\doteq(F_Y\cup H^4,F_M\cup H^2, h_Y \circ \tau^{-1}_L,h_M \circ \tau^{-1}_K)$$
		where this time the framing of $H^4$ is one more than the contact framing of $\partial A$ in $(\partial F_Y,\textrm{Ker}(\alpha |_{\,\partial F_Y}))$.
	\end{definition}
	
	\begin{remark}
		From the definition it is understood that a relative stabilization consists of two stabilizations one of which performed in dimension three and the other in five, but there is only one stabilization indeed: we are only stabilizing the open book $(F_Y,h_Y)$ along $A$, i.e., 
		$$\mathcal{S}^\pm[(F_Y,F_M,h_Y,h_M);A, A\cap F_M]=\mathcal{S}^\pm[(F_Y,h_Y);A],$$ and as a result of this $(F_M,h_M)$ is also simultaneously stabilized along $A\cap F_M$. 
	\end{remark}
	
	The result of Giroux holds also in the relative setting. Namely,
	
	\begin{theorem}  \label{thm:Relative_Stab_Compatible_Open Book}
		In the setting of Definition \ref{def:Stab_of_Rel_Open_Book}, both $\mathcal{S}^\pm[(F_Y,F_M,h_Y,h_M);A, A\cap F_M]$ are open books on the connected sum $(Y\# \mathbb{S}^{5},M\# \mathbb{S}^{3})\cong (Y,M)$. Indeed, $\mathcal{S}^+[(F_Y,F_M,h_Y,h_M);A, A\cap F_M]$ is an open book on the contact connected sum $(Y_\xi\# \mathbb{S}^{5}_{st},M_\eta \# \mathbb{S}^{3}_{st})\cong (Y_\xi,M_\eta)$, in other words, it is a compatible relative open book on the relative contact pair $(Y_\xi,M_\eta)$.
	\end{theorem} 
	
	\begin{proof}
		Observe that both stabilizations $\mathcal{S}^\pm[(F_Y,F_M,h_Y,h_M);A, A\cap F_M]$ correspond to the stabilized open books $\mathcal{S}^\pm[(F_Y,h_Y);A]$ on $Y\# \mathbb{S}^{5}\cong Y$ and $\mathcal{S}^\pm[(F_M,h_M);A\cap F_M]$ on $M\# \mathbb{S}^{3}\cong M$. Therefore, by Theorem \ref{thm:Stabilization_Compatible_Open Book}, both are relative open books on the relative connected sum $(Y\# \mathbb{S}^{5},M\# \mathbb{S}^{3})\cong (Y,M)$, and also $\mathcal{S}^+[(F_Y,F_M,h_Y,h_M);A, A\cap F_M]$ is an open book on the contact connected sum $(Y_\xi\# \mathbb{S}^{5}_{st},M_\eta \# \mathbb{S}^{3}_{st})\cong (Y_\xi,M_\eta)$. Hence, Definition \ref{def:Compatibility_of_Rel_Open_Book} applies.
	\end{proof}
	
	\begin{example} \label{ex:basic_stabilization}
		Consider the simplest compatible open books $(D^4, id_{D^4})$ and $(D^2,id_{D^2})$ on the standard contact spheres $\mathbb{S}^5_{st}$ and $\mathbb{S}^3_{st}$, respectively. Here we consider both disks with their standard symplectic structures, and $D^2$ as a properly and symplectically embedded in $D^4$. One easily checks the conditions in Definition \ref{def:Abstract_Rel_Open_Book} and Definition \ref{def:Compatibility_of_Rel_Open_Book}, and so $(D^4, D^2, id_{D^4}, id_{D^2})$ is a compatible relative open book on the contact pair $(\mathbb{S}^5_{st},\mathbb{S}^3_{st})$. Also it is a well-known construction that $D^*\mathbb{S}^2$ and $D^*\mathbb{S}^1$ can be, simultanaously, obtained from $D^4$ and $D^2$ by attachhing Weinstein handles $H^4$ and $H^2$ along a Legendrian $1$-unknot $\partial A$ (the boundary of a properly embedded Lagrangian disk $A$ in $D^4$) in $\mathbb{S}^3_{st}(=\partial D^4)$ and a Legendrian $0$-unknot $\partial (L\cap D^2)$ in $\mathbb{S}^1_{st}(=\partial D^2)$. Assuming the framing of $H^4$ is chosen correctly (one less than the contact framing of $\partial A$ in $\mathbb{S}^3_{st}=\partial D^4$), we conclude, by Definition \ref{def:Stab_of_Rel_Open_Book}, that
		$(D^*\mathbb{S}^2,D^*\mathbb{S}^1,\tau_{\mathbb{S}^2},\tau_{\mathbb{S}^1})$ is a positive relative stabilization of $(D^4, D^2, id_{D^4}, id_{D^2})$ along $(A,A\cap D^2)$. That is, $$(D^*\mathbb{S}^2,D^*\mathbb{S}^1,\tau_{\mathbb{S}^2},\tau_{\mathbb{S}^1})=\mathcal{S}^+[(D^4, D^2, id_{D^4}, id_{D^2});A, A\cap D^2].$$
		On the other hand, if the framing of $H^4$ is chosen to be one more than the contact framing of $\partial A$, then one concludes that 
		$$(D^*\mathbb{S}^2,D^*\mathbb{S}^1,\tau^{-1}_{\mathbb{S}^2},\tau^{-1}_{\mathbb{S}^1})=\mathcal{S}^-[(D^4, D^2, id_{D^4}, id_{D^2});A, A\cap D^2].$$
		
	\end{example}

	
	\section{Generalized Square Bridge Diagrams} \label{sec:generalized-square-bridge}
	
	Any Legendrian link $\mathbb{K}$ in $(\mathbb{R}^3,\xi_0^3)\subset \mathbb{S}^3_{st}$ can be represented by a square bridge diagram \cite{Ly}. In this section, we will construct \textit{a generalized square bridge diagram} to represent any admissible Legendrian $2$-link $\mathbb{L}$ in $(\mathbb{R}^5,\xi_0^5)\subset \mathbb{S}^5_{st}$ which naturally contains a square bridge diagram of the equatorial link $\mathbb{K}=\mathbb{L}\cap\mathbb{S}^5_{st}$ of $\mathbb{L}$. \\
	
	Let  $\Pi_5:\mathbb{R}^5\rightarrow \mathbb{R}^3: (x_1,y_1,z,x_2,y_2)\mapsto (y_1,z,y_2)$ and $ \Pi_3:  \mathbb{R}^3\rightarrow \mathbb{R}^2: (x_1,y_1,z)\mapsto (y_1,z)$ be the \textit{front projection} maps. In order to picture Legendrian $2$-knots (resp. Legendrian $1$-knots) in $(\mathbb{R}^5,\xi_0^5)$ (resp. $(\mathbb{R}^3,\xi_0^3)$), we consider their \textit{fronts}, i.e., their images under the map $ \Pi_5$ (resp. $ \Pi_3$). Examples of such fronts are depicted for Legendrian unknots in Figure \ref{Fig:fronts}.
	
	\begin{figure}[H]
		\begin{center}
			\includegraphics[width=0.8\textwidth]{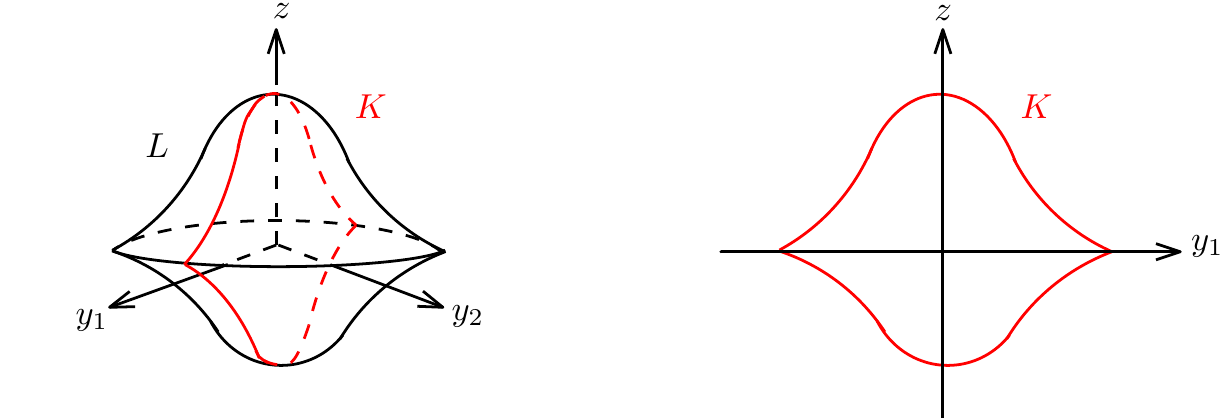}
			\caption{On the left, the front of the admissible Legendrian $2$-unknot $\mathbb{L}$, and, on the right, the front of its isotropic equator $\mathbb{K}$ (Legendrian $1$-unknot in $(\mathbb{R}^3,\xi_0^3)$.}
			\label{Fig:fronts}
		\end{center}
	\end{figure}
	
	\indent We first consider the square bridge diagram (in the $y_1z$-plane) of the front of the equatorial $1$-link $\mathbb{K}=\bigsqcup_{i=1}^{r} K_{i} $ of a given admissible Legendrian $2$-link $\mathbb{L}=\bigsqcup_{i=1}^{r} L_{i}$ such that it consists of the following two types of line segments:
	
	\begin{eqnarray*}
		h_{a_i}&=&\{(x_1,y_1,z,x_2,y_2)\in\mathbb{R}^5:\quad x_1=x_2=y_2=0,\quad z=-y_1+a_i \} \\
		v_{b_j}&=&\{(x_1,y_1,z,x_2,y_2)\in\mathbb{R}^5:\quad x_1=x_2=y_2=0,\quad  z=y_1+b_j \}
	\end{eqnarray*}
	for some real numbers $a_1< a_2<\cdots<a_p$ and  $0< b_1<b_2<\cdots< b_q$, and it bounds a polygonal region in the $y_1z-$plane as in Figure \ref{Fig:sbd of Hopf link} (on the left). Secondly, divide each polygonal region into finitely many square regions $R_k$ for $k=1,2,. . . ,m$ as in Figure \ref{Fig:sbd of Hopf link} (on the right). Such a division can be achieved by applying an appropriate Legendrian isotopy to $\mathbb{L}$ so that the gaps between consequtive $a_i$'s and consequtive $b_j$'s are all equal to the same fixed real number.
	
	\begin{figure}[H]
		\begin{center}
			\includegraphics[width=0.7\textwidth]{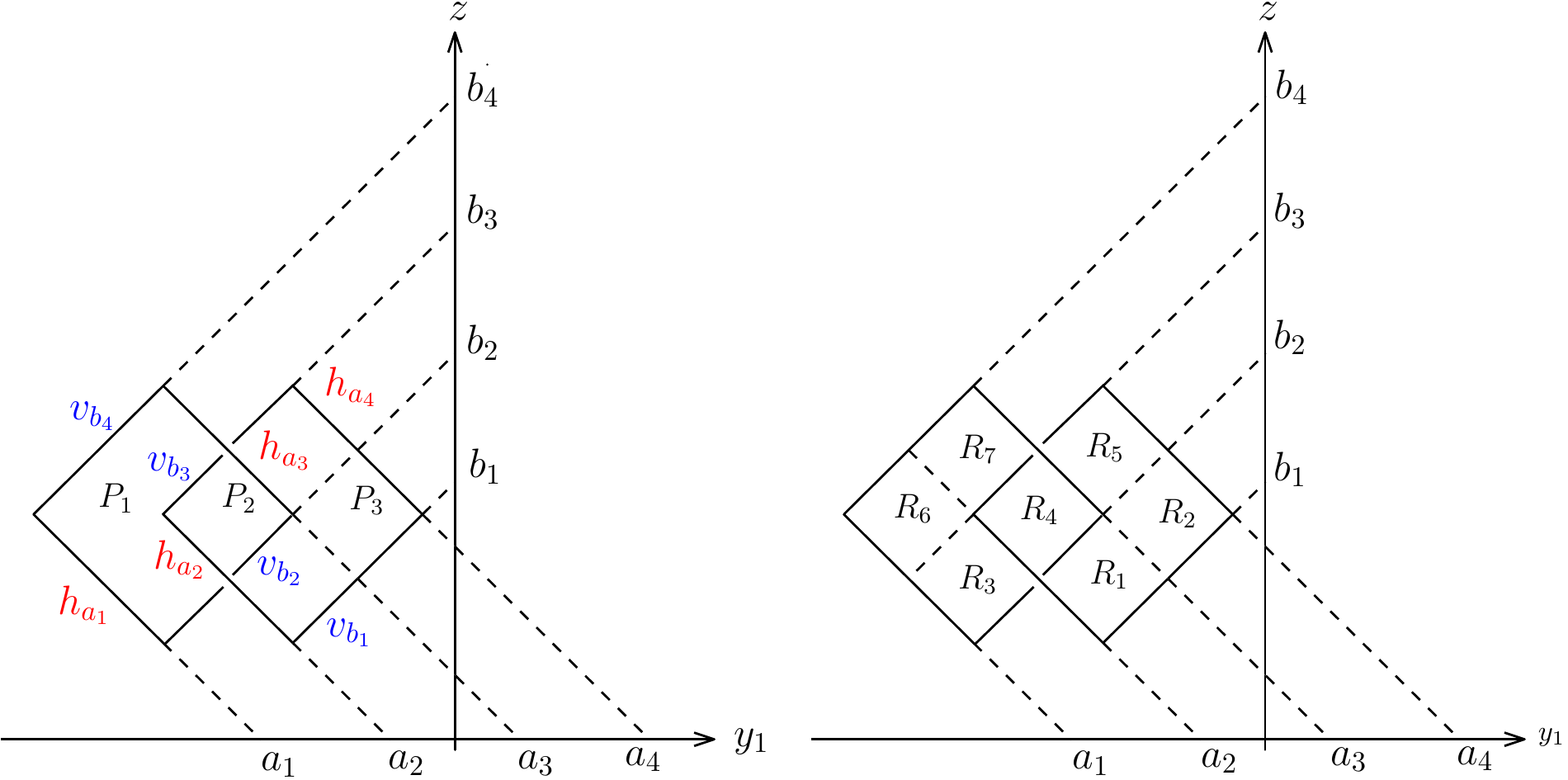}
			\caption{The square bridge diagram of the equatorial isotropic $1$-link (demonstrated when $\mathbb{K}$ is the Hopf link) and its division into squares (on the right).}
			\label{Fig:sbd of Hopf link}
		\end{center}
	\end{figure}
	\indent Thirdly, we construct an octahedron for each square region $R_k$ so that $\partial R_k$ is the equatorial square of the octahedron. This can be done by composing (gluing) eight triangular plane segments of the following four types:
	
	\begin{eqnarray*}
		t_{a_i}^{1,1}\;&=&\{(x_1,y_1,z,x_2,y_2)\in\mathbb{R}^5:\quad x_1=x_2=0,\quad z=-y_1-y_2+ a_i \}\\
		t_{a_i}^{1,-1}\;&=&\{(x_1,y_1,z,x_2,y_2)\in\mathbb{R}^5:\quad x_1=x_2=0,\quad z=-y_1+y_2+ a_i \}\\
		t_{b_j}^{-1,1}&=&\{(x_1,y_1,z,x_2,y_2)\in\mathbb{R}^5:\quad x_1=x_2=0,\quad z=y_1-y_2+ b_j \}\\
		t_{b_j}^{-1,-1}&=&\{(x_1,y_1,z,x_2,y_2)\in\mathbb{R}^5:\quad x_1=x_2=0,\quad z=y_1+y_2+ b_j \}.
	\end{eqnarray*}
	
	A typical octahedron for a typical square region $R_k$ is depicted in Figure \ref{Fig:octahedron} which is indeed what we call as the \textit{generalized square bridge position} of a standard Legendrian $2$-unknot whose front has appeared in Figure \ref{Fig:fronts} above.
	\begin{figure}[H]
		\includegraphics[width=0.65\textwidth]{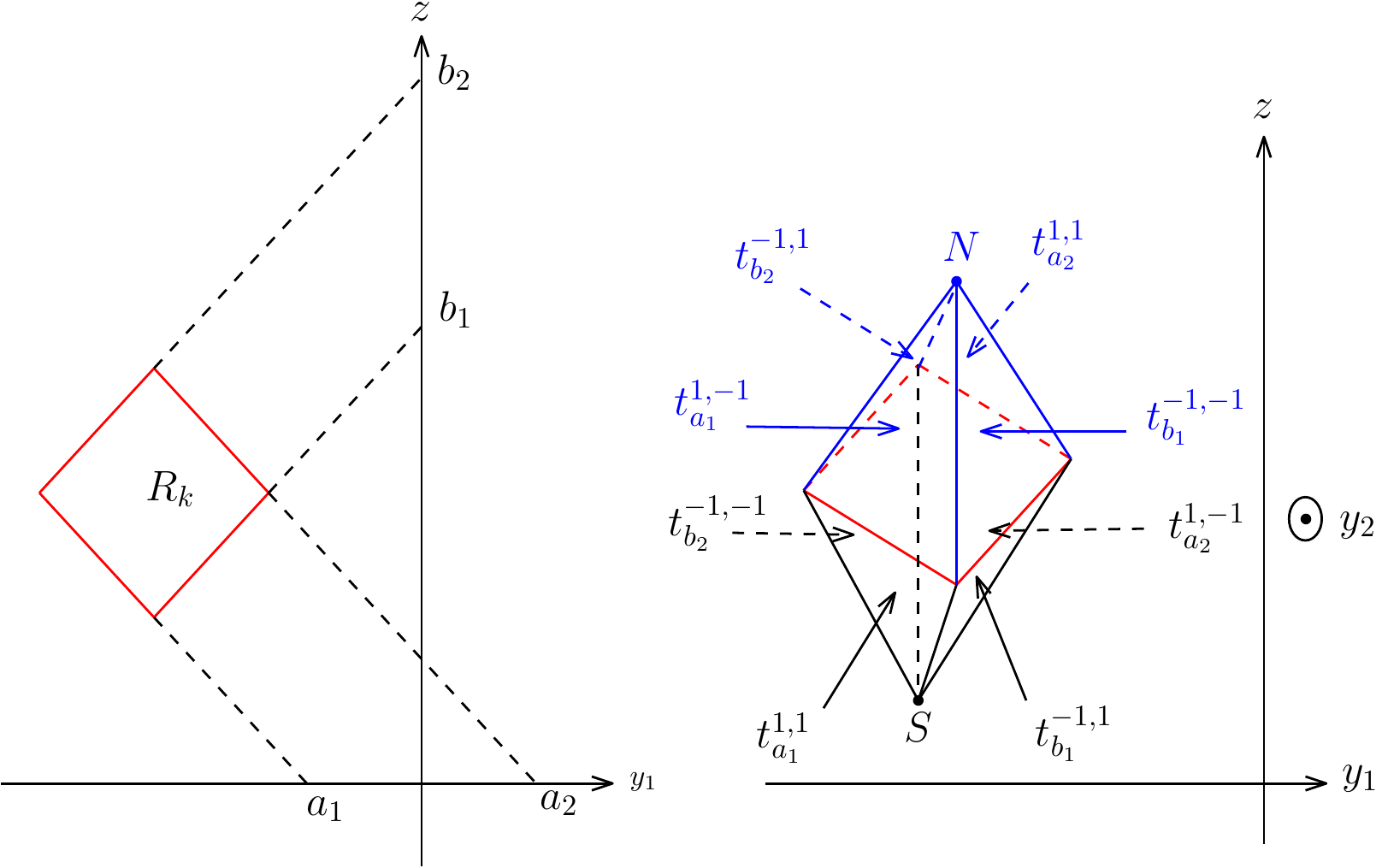}
		\caption{An octahedron corresponding to $R_k$ (a generalized square bridge position of a standard Legendrian $2$-unknot.}
		\label{Fig:octahedron}
	\end{figure}
	
	In Figure \ref{Fig:octahedron}, the upper part of the octahedron contains the triangles $t_{a_1}^{1,-1}$, $t_{a_2}^{1,1}$, $t_{b_1}^{-1,-1}, t_{b_2}^{-1,1}$, and the lower part contains the triangles $t_{a_1}^{1,1}$, $t_{a_2}^{1,-1}$, $t_{b_1}^{-1,1}, t_{b_2}^{-1,-1}$. Also the north and south poles of the octahedron are $\displaystyle N=(0, \frac{a_1-b_1}{2}, \frac{a_1+b_2}{2}, 0, \frac{b_2-b_1}{2} )$ and $\displaystyle S=(0, \frac{a_1-b_1}{2}, \frac{a_1+b_2}{2}, 0, \frac{b_1-b_2}{2} )$.\\
	
	Finally, we define
	\begin{definition}
		The \textit{generalized square bridge diagram}, denoted by $\Delta(\mathbb{L})$, of $\mathbb{L}$ is a cellular $2$-complex in the $(y_1z\,y_2)$-space obtained by considering all $R_k$'s (in the square bridge diagram of $\mathbb{K}$) and taking the union of corresponding $m$ octahedrons (see Figure \ref{Fig:Hopf link}).
	\end{definition}
	
	\begin{figure}[H]
		\begin{center}
			\includegraphics[width=0.4\textwidth]{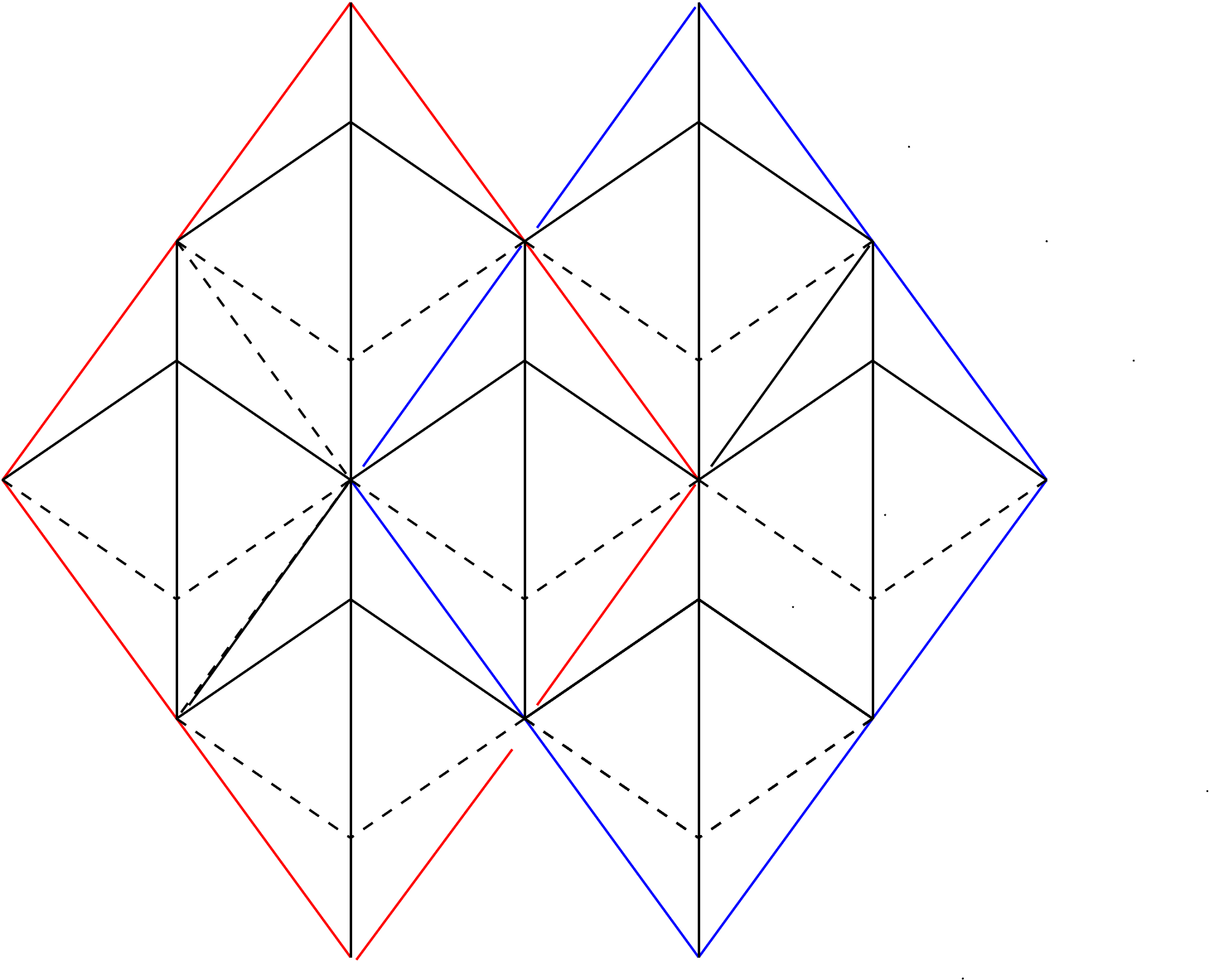}
			\caption{The generalized square bridge diagram $\Delta(\mathbb{L})$ of the admissible Legendrian $2$-link $\mathbb{L}$ whose equatorial link $\mathbb{K}$ is the Hopf link.}
			\label{Fig:Hopf link}
		\end{center}
	\end{figure}
	
	Note that the square bridge diagram $\Gamma(\mathbb{K})$ of $\mathbb{K}$ is a cellular $1$-subcomplex of $\Delta(\mathbb{L})$.  Next, we will show that starting from its generalized square bridge diagram $\Delta(L)\subset \Delta(\mathbb{L})$, one can reconstruct each component $L$ of $\mathbb{L}$ in $(\mathbb{R}^5,\xi_0^5)\subset \mathbb{S}^5_{st}$. More precisely,
	
	\begin{proposition} \label{Prop:reconstruction}
		For any component $L_i$ of \,$\mathbb{L}=\bigsqcup_{i=1}^{r} L_{i}$, starting from $\Delta(L_i)$, one can construct a piecewise smooth Legendrian $2$-sphere $L'_i$ in $(\mathbb{R}^5,\xi_0^5)\subset \mathbb{S}^5_{st}$ by gluing appropriate Legendrian and planar triangles and rectangles so that the smooth Legendrian $2$-sphere  obtained by rounding the corners of $L_i'$ is Legendrian isotopic to $L_i$. Moreover, the piecewise smooth equatorial $1$-link $\mathbb{K}'=\bigsqcup_{i=1}^{r} K'_{i}$ ($K'_{i}=L'_i \cap\mathbb{S}^3_{st}$) and $\mathbb{K}$ have the same linking matrix.
	\end{proposition}
	
	\begin{proof} 
		We first lift each line segment $h_{a_i}$ and $v_{b_j}$ in $\Gamma(\mathbb{K})$ (divided into square regions) along the $x_1$-axis such that the lifts are disjoint Legendrian line segments in $(\mathbb{R}^3,\xi_0^3)\subset \mathbb{S}^3_{st}$ given by
		
		\begin{eqnarray*}
			H_{a_i}&=&\{(x_1,y_1,z,x_2,y_2)\in\mathbb{R}^5:\quad x_1=1, x_2=y_2=0,\quad z=-y_1+a_i \}\\
			V_{b_j}&=&\{(x_1,y_1,z,x_2,y_2)\in\mathbb{R}^5:\quad x_1=-1, x_2=y_2=0,\quad z=y_1+b_j \}.
		\end{eqnarray*}
		
		Then we connect each segment $H_{a_i}$  to certain segments $V_{b_j}$ via Legendrian linear arcs $I_k^l$ which are parallel to the $x_1$-axis so that all of the (piecewise smooth) Legendrian unknots $\gamma_k$ in $\Gamma(\mathbb{K})$ are obtained in $(\mathbb{R}^3,\xi_0^3)$ as in \cite{Ar}. The construction of $\gamma_k$ (for a typical square region $R_k$) is illustrated in Figure \ref{Fig:Legendrian unknot}. Note that by construction the Legendrian unknot $\gamma_k$ is attached to $\gamma_{k+1}$ by either sidewise or by opening a back gate. As a result, we have constructed the lift $\widetilde{\Gamma}(\mathbb{K})$ of $\Gamma(\mathbb{K})$ as a Legendrian cellular $1$-complex in $(\mathbb{R}^3,\xi_0^3)\subset \mathbb{S}^3_{st}$ (which is also an isotropic cellular $1$-complex in $(\mathbb{R}^5,\xi_0^5)\subset \mathbb{S}^5_{st}$). Observe that by construction there exists a piecewise smooth Legendrian $1$-link $\mathbb{K}'=\bigsqcup_{i=1}^{r} K'_{i} \subset \widetilde{\Gamma}(\mathbb{K})$ which we can round its corners (in its $\epsilon$-neighborhood) so that the resulting smooth Legendrian link is Legendrian isotopic to $\mathbb{K}$. In particular, the linking matrices of $\mathbb{K}'$ and $\mathbb{K}$ are the same. Hence, the second claim in the statement will be satisfied once we prove the first one (i.e., once we construct $L_i'$) in such a way that $K'_{i}=L'_i \cap\mathbb{S}^3_{st}$. 
		
		\begin{figure}[H]
			\includegraphics[width=0.6\textwidth]{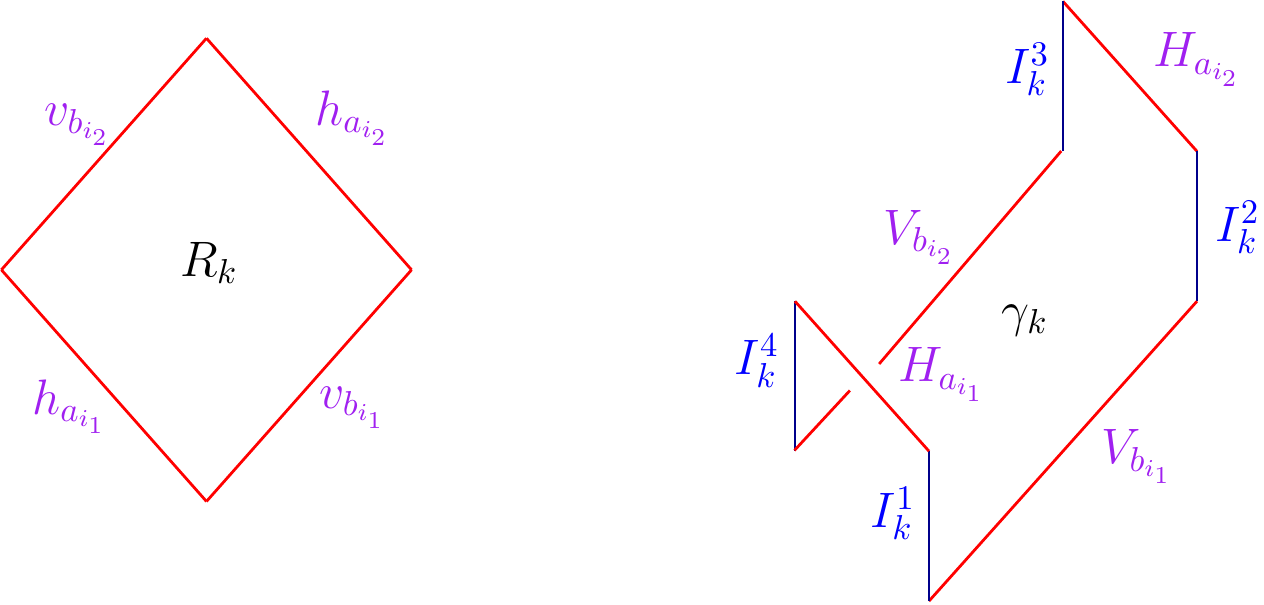}
			\caption{A typical square region $R_k$ and the Legendrian unknot $\gamma_k$.}
			\label{Fig:Legendrian unknot}
		\end{figure}
		
		Next we will construct a Legendrian thickening $\mathcal{L}(\widetilde{\Gamma}(\mathbb{K}))$ of  $\widetilde{\Gamma}(\mathbb{K})$ (the union of all $\gamma_k$'s) which will contain the neighborhood of each (equatorial knot) $K_i'$ inside Legendrian $L_i'$ that we will eventually construct. Although there is no unique way of Legendrian thickening (need to choose a thickening direction), we do this in a systematic way as follows: For each $k=1, ..., m$, we first construct the Legendrian thickening $M_k$ (which we call a \textit{middle part}) of $\gamma_k$ in the way explained below. Then we form 
		$\mathcal{L}(\widetilde{\Gamma}(\mathbb{K}))\doteq\bigsqcup_{k=1}^{m} M_k$ where the union is taken over the common parts of $M_k$'s (By the construction below, if two $\gamma_k$'s have a common edge, the corresponding $M_k$'s have a common rectangular part which is the Legendrian thickening of the common edge.) \\ 
		
		We construct each middle part $M_k$ so that it consists of 4 Legendrian rectangular plates and 16 Legendrian triangular ones as shown in the Figure \ref{Fig:Middle part}. The whole $M_k$ lies in the hyperspace $\{y_2=0\}\subset (\mathbb{R}^5,\xi_0^5)$, and $\gamma_k$ lies as the core of $M_k$ where $x_2=y_2=0$.
		
		\begin{figure}[H]
			\begin{center}
				\includegraphics[width=0.7\textwidth]{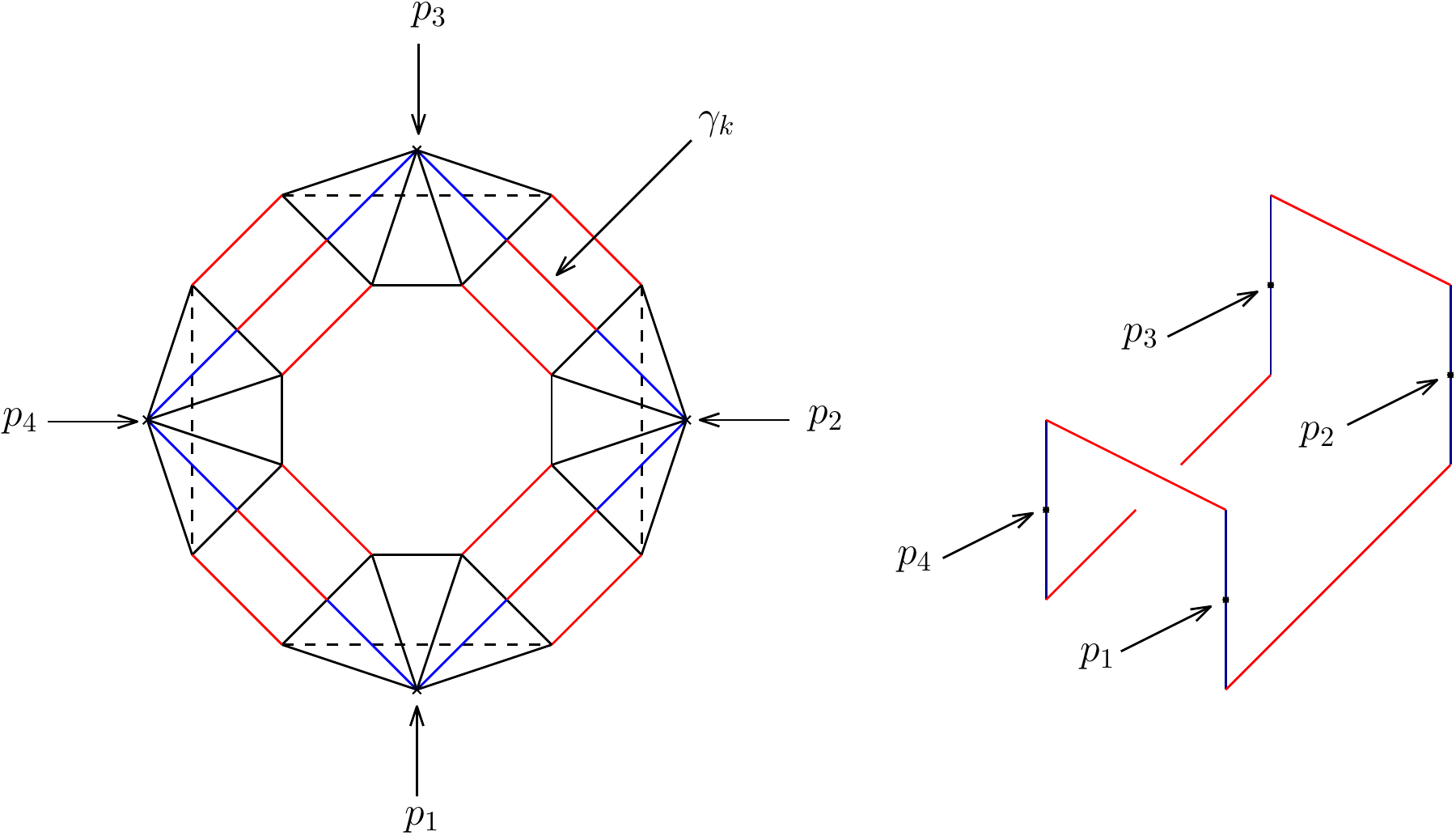}
				\caption{A shematic picture of the middle part $M_k$ for a typical square $R_k$ (on the left), and the isotropic unknot $\gamma_k$ (on the right).}
				\label{Fig:Middle part}
			\end{center}
		\end{figure}
		
		For $H_i$ and $V_j$ parts of $\gamma_k$, we take their Legendrian thickenings to be the $4$ Legendrian rectangles $S_{{H}_{i}}$ and $S_{{V}_{j}}$ on which $x_1=\pm 1$ and $-1\leq x_2 \leq 1$. On the other hand, for ${I}_{k}$ parts of $\gamma_k$, their Legendrian thickenings totally consist of $16$ Legendrian triangles $S_{{I}_{k}}$ on which both $x_1$ and $x_2$ change between $-1$ and $1$ in such a way that each group of $4$ triangles together join an $S_{{H}_{i}}$ to one of the two of its neighboring $S_{{V}_{j}}$'s. In particular, $x_1=x_2=0$ at the points $p_1$, $p_2$, $p_3, p_4$ of $M_k$ which are the midpoints of the arcs ${I}_{k}$ in $\gamma_k$. (See Figure \ref{Fig: Middle part 2} for a shematic picture of this construction performed in a four dimensional space $\{y_2=0\}\subset (\mathbb{R}^5,\xi_0^5)$.)
		
		\begin{figure}[H]
			\begin{center}
				\includegraphics[width=0.75\textwidth]{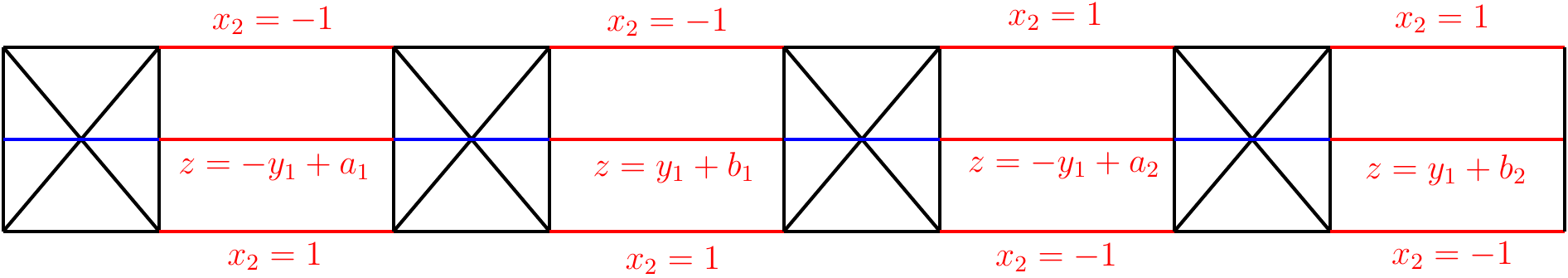}
				\caption{A shematic picture of the construction of $M_k$. (The most left and right vertical edges are identified.)}
				\label{Fig: Middle part 2}
			\end{center}
		\end{figure}
		
		By construction, all $M_k$'s touch together along their common parts, and they together form the Legendrian thickening $\mathcal{L}(\widetilde{\Gamma}(\mathbb{K}))\doteq\bigsqcup_{k=1}^{m} M_k$. We note that $\mathbb{K}' \subset \widetilde{\Gamma}(\mathbb{K}) \subset  \mathcal{L}(\widetilde{\Gamma}(\mathbb{K}))$.\\
		
		Now we consider the triangles of the octahedrons in the generalized square bridge diagram $\Delta(\mathbb{L})$. We lift each of the triangular plane segments $t_{a_i}^{1,1}$, $t_{a_i}^{1,-1}$, $t_{b_j}^{-1,1}$, $t_{b_j}^{-1,-1}$ (to an appropriate $x_1$ and $x_2$ levels) so that the resulting lifts are disjoint Legendrian triangular plates lying on the following $4$-types of planes:
		
		$$\begin{array}{rclll}
			T_{a_i}^{1,1}\;\;&=&\{(x_1,y_1,z,x_2,y_2)\in\mathbb{R}^5:\quad x_1=1,\quad &x_2 = 1, \quad &z=-y_1-y_2+ a_i \}\\
			T_{a_i}^{1,-1}\;&=&\{(x_1,y_1,z,x_2,y_2)\in\mathbb{R}^5:\quad x_1=1,\quad &x_2 =-1, \quad &z=-y_1+y_2+ a_i \} \\
			T_{b_j}^{-1,1}\;&=&\{(x_1,y_1,z,x_2,y_2)\in\mathbb{R}^5:\quad x_1=-1,\quad &x_2 = 1, \quad &z=y_1-y_2+ b_j \} \\
			T_{b_j}^{-1,-1}&=&\{(x_1,y_1,z,x_2,y_2)\in\mathbb{R}^5:\quad x_1=-1,\quad &x_2 = -1, \quad &z=y_1+y_2+ b_j \} .
		\end{array}$$
		
		The next lemma shows that the triangular planes listed above are all Legendrian.
		\begin{lemma}
			For any $c\in \mathbb{R}$, the planes $$T_c^{k,\ell}=\{(x_1,y_1,z,x_2,y_2)\in\mathbb{R}^5:\quad x_1=m,\quad x_2=n,\quad z=ky_1+\ell y_2+c\}$$ are Legendrian in $(\mathbb{R}^5,\xi_0^5)\subset \mathbb{S}^5_{st} $ \;if and only if\; $k=-m$ \;and\; $\ell =-n$.
		\end{lemma}
		\begin{proof}
			Need to show any vector tangent to $T_c^{k,\ell}$ lies in $\textrm{Ker}(\alpha_0^5)$ where 
			$\alpha_0^5=dz+x_1dy_1+x_2dy_2$. The tangent space of $T_c^{k,\ell}$ at any point $p\in T_c^{k,\ell}$ is spanned by the 
			vectors $\partial / \partial y_1+k\,\partial /\partial z$ and $\partial /\partial y_2+\ell \,\partial /\partial z$. 
			That is,  $$T_pT_c^{k,\ell}=\left\langle \frac{\partial}{\partial y_1}+k\frac{\partial}{\partial z}, \;\frac{\partial}{\partial y_2}+\ell \frac{\partial}{\partial z}\right\rangle.$$ 
			Then one can easily verify that   
			$$\alpha_0^5\left(\frac{\partial}{\partial y_1}+k\frac{\partial}{\partial z}\right)=k+m=0 \Leftrightarrow k=-m, \quad \alpha_0^5\left(\frac{\partial}{\partial y_1}+\ell\frac{\partial}{\partial z}\right)=\ell +n=0 \Leftrightarrow \ell =-n .$$
			Hence, the claim follows from linearity.
		\end{proof}
		
		Next, using the above triangles, for each square region $R_k$ we will obtain two disjoint piecewise smooth disks $\mathcal{D}^{'}_k$ and $\mathcal{D}^{''}_k$ which will be the upper and lower part of a piecewise smooth Legendrian $2$-unknot, called a \textit{Legendrian diamond} in $(\mathbb{R}^5, \xi_0^5)$  as shown in the Figure \ref{Fig: Upper and Lower part} and Figure \ref{Fig: Diamond}.
		
		\begin{figure}[H]
			\begin{center}
				\includegraphics[width=0.7\textwidth]{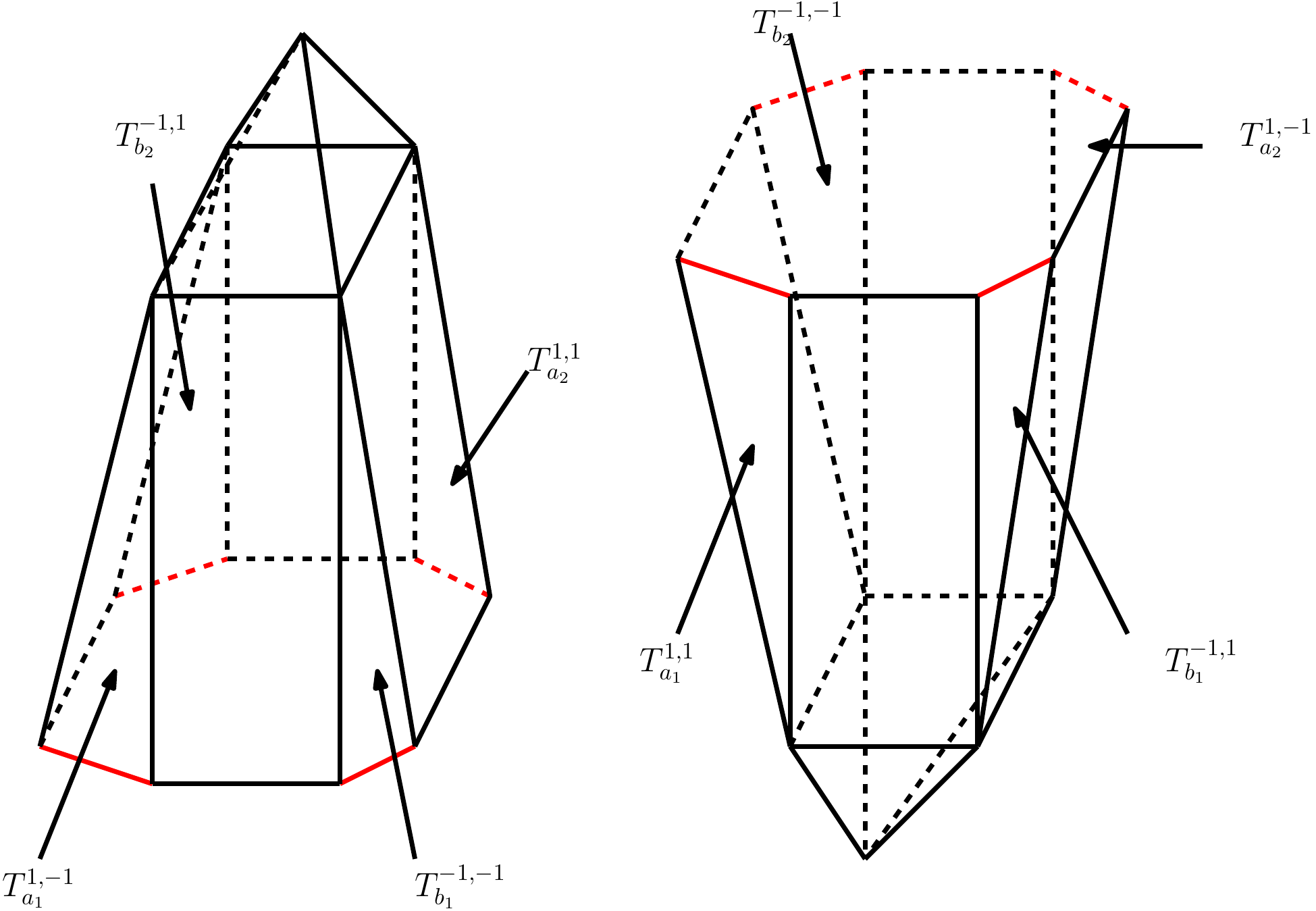}
				\caption{The Legendrian piecewise smooth disk $\mathcal{D}^{'}_k$(on the left) and the Legendrian piecewise smooth disk $\mathcal{D}^{''}_k$  (on the right).}
				\label{Fig: Upper and Lower part}
			\end{center}
		\end{figure}
		
		The disk $\mathcal{D}^{'}_k$ consists of $4$ disjoint Legendrian triangular plates $T_{a_1}^{1,-1}$, $T_{a_2}^{1,1}$, $T_{b_1}^{-1,-1}$,$T_{b_2}^{-1,1}$, $4$ Legendrian rectangular plates (connecting them), and $4$ other Legendrian triangular plates at the top, and likewise, the  disk $\mathcal{D}^{''}_k$ consists of $4$ disjoint Legendrian triangular plates $T_{a_1}^{1,1}$,  $T_{a_2}^{1,-1}$, $T_{b_1}^{-1,1}$, $T_{b_2}^{-1,-1}$, $4$ Legendrian rectangular plates (connecting them), and $4$ other Legendrian triangular plates at the botton (Figure \ref{Fig: Upper and Lower part}). For each disk, the rectangular plates connect certain two of the disjoint triangular plates, and  the four triangular plates connecting the four Legendrian rectagular plates are the following: 
		\begin{align*}
			&{\tiny P_{\pm 1}=\{z=\pm y_2+ \frac{a_i+b_i}{2},\quad y_1=\frac{a_i-b_i}{2},\quad  x_2=\mp 1,\quad -1 \leq x_1 \leq 1\}}\\
			& P^{'}_{\pm 1}=\{y_2=\pm y_1\pm \frac{b_2-a_1}{2}, \quad z=\frac{a_1+b_2}{2} \quad -1 \leq x_1=\mp x_2 \leq 1\}.
		\end{align*}
		\indent Also, for each disk, there are two plates are composed of $4$-triangular plates on the top and bottom:
		\begin{align*}
			&P_{t}=\{ y_1=\frac{a_1-b_1}{2},\quad z=\frac{a_1+b_2}{2},\quad y_2= \frac{b_2-b_1}{2}, \quad  -1 \leq x_1, x_2 \leq 1\}\\
			&P_{b}=\{ y_1=\frac{a_1-b_1}{2},\quad z=\frac{a_1+b_2}{2},\quad y_2= \frac{b_1-b_2}{2}, \quad  -1 \leq x_1, x_2 \leq 1\} .
		\end{align*}
		
		\begin{figure}[H]
			\begin{center}
				\includegraphics[width=0.3\textwidth]{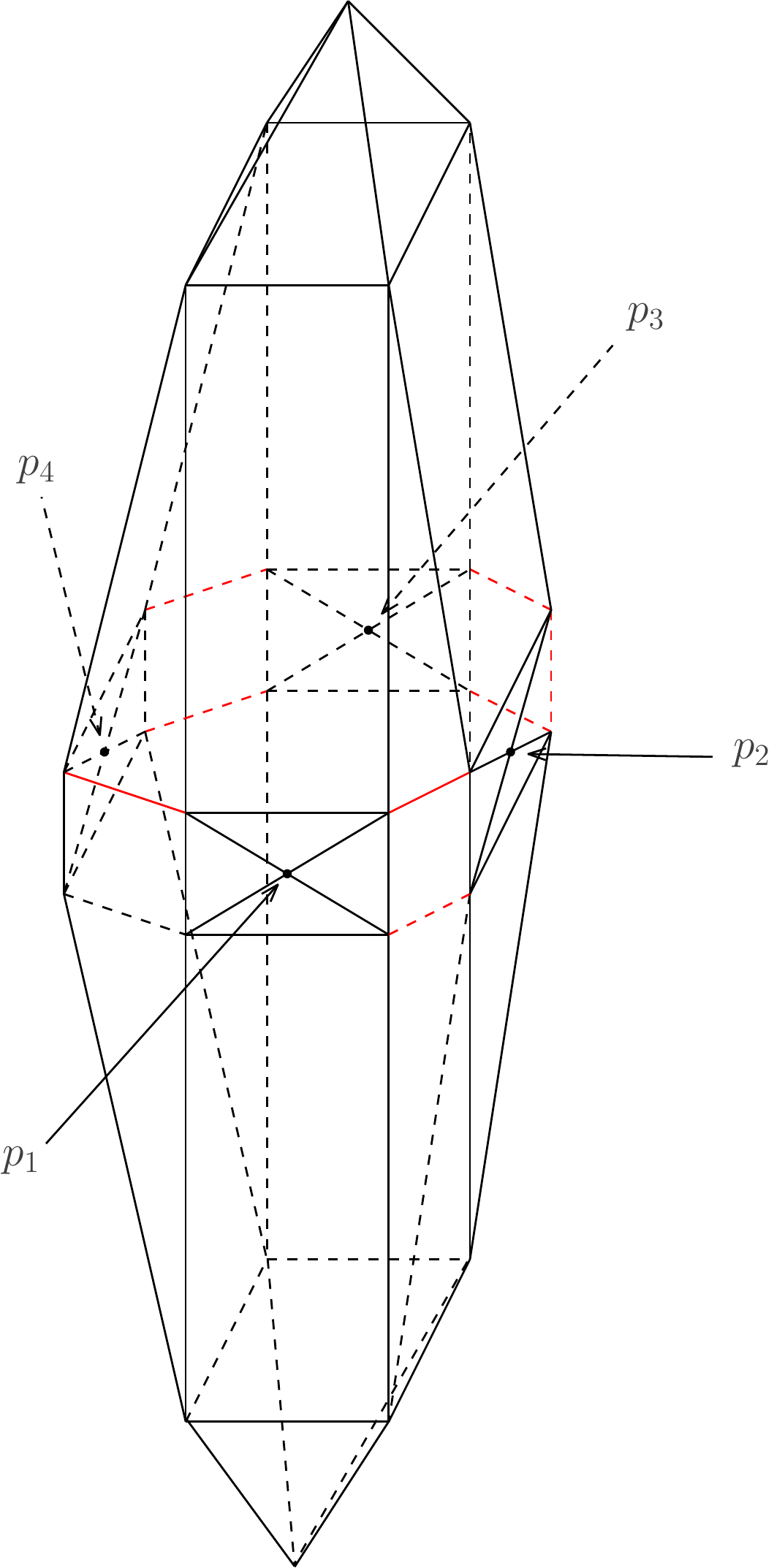}
				\caption{The Legendrian diamond $\mathcal{D}_k$ (a piecewise smooth Legendrian $2$-unknot).}
				\label{Fig: Diamond}
			\end{center}
		\end{figure}
		From their constructions, it is straight forward to check that all the triangular and rectangular plates in $\mathcal{D}'_k$, $\mathcal{D}''_k$ and $M_k$ meet each other along the common parts of their boundaries appropriately, and they together form the Legendrian diamond $\mathcal{D}_k\doteq \mathcal{D}^{'}_k\cup M_k\cup \mathcal{D}^{''}_k$ in $(\mathbb{R}^5,\xi_0^5)\subset \mathbb{S}^5_{st}$. We note that the equatorial link $\mathbb{K}=\bigsqcup_{i=1}^{r} K_{i}$ is reconstructed as the piecewise smooth $1$-link $\mathbb{K}'=\bigsqcup_{i=1}^{r} K'_{i}$ contained in the union $\mathcal{L}(\widetilde{\Gamma}(\mathbb{K}))(=\bigsqcup_{k=1}^{m}M_k)$ which can be realized as the union of the neighborhoods of the isotropic equatorial unknots inside the union $\bigsqcup_{k=1}^{m} \mathcal{D}_k$ of all Legendrian diamonds.\\
		
		To verify the first claim, let $L_i$ be any component of $\mathbb{L}$. Consider the subcollection of all square regions $\{R_{i_1},R_{i_2}, ..., R_{i_n}\} \subset \{R_{1},R_{2}, ..., R_{m}\}$ which are contained in the square bridge diagram $\Gamma(K_i) \subset \Gamma(\mathbb{K})$. Then the front $\Pi_3(K_i)$ is contained in the union $\bigsqcup_{j=1}^{n} R_{i_j}$, and so the corresponding piecewise smooth isotropic unknot $K_i'$ is contained in the Legendrian thickening 
		$\mathcal{L}(\widetilde{\Gamma}(K_i))=\bigsqcup_{j=1}^{n} M_{i_j} \subset \bigsqcup_{j=1}^{n} \mathcal{D}_{i_j}$. Let $\mathcal{L}(K_i')\subset \mathcal{L}(\widetilde{\Gamma}(K_i))$ denote the Legendrian thickening of $K_i'$ (i.e., the union of the plates in $\bigsqcup_{j=1}^{n} M_{i_j}$ which retracts onto $K_i'$). The boundary $\partial (\mathcal{L}(K_i'))$ consists of two (isotropic) push-offs of $K_i'$ where $\mathcal{L}(K_i')$ meets with the corresponding rectangular and triangular plates in the upper and lower disks $\mathcal{D}^{'}_{i_j}$, $\mathcal{D}^{''}_{i_j}$ of the Legendrian diamonds $\mathcal{D}_{i_1},\mathcal{D}_{i_2}, ..., \mathcal{D}_{i_n}$.
		Finally, we define $L_i'$ as the union
		$$L_i'\doteq \mathcal{L}(K_i') \cup_{\partial (\mathcal{L}(K_i'))} \bigsqcup_{j=1}^{n}(\overline{\mathcal{D}_{i_j}\setminus M_{i_j}}).$$
		Clearly, $L_i'$ is a piecewise smooth Legendrian $2$-unknot in $\mathbb{S}^5_{st}$ with its equatorial $1$-unknot $K_i'$ in $\mathbb{S}^3_{st}$ as required. After rounding the corners (and sharp edges) in $L_i'$ (and so, in $K_i'$) in a small neighborhood of $L_i'$, we obtain a smooth Legendrian admissible $2$-unknot, say $L_i''$, with its equatorial $1$-unknot $K_i''$ which is Legendrian isotopic to the equatorial $1$-unknot $K_i$ of $L_i$. By extending this isotopy over the union of the upper and the lower Legendrian disks of $L_i''$ (i.e., over $L_i''\setminus K_i''$), we conclude that $L_i''$ and $L_i$ are Legendrian isotopic in $\mathbb{S}^5_{st}$.
		This finishes the proof of Proposition \ref{Prop:reconstruction}.
	\end{proof}
	
	\begin{remark}
		As observed in the proof of Proposition \ref{Prop:reconstruction}, the linking matrices of the links $\mathbb{K}$, $\mathbb{K}'$, and so $\mathbb{K}''\doteq\bigsqcup_{i=1}^{r} K_{i}''$ are all the same. Also if two line segments $h_{a_i}$, $v_{b_j}$ cross each other in the square bridge diagram $\Gamma(\mathbb{K})$,  the lift $H_{a_i}$ of $h_{a_i}$ crosses over the lift $V_{b_j}$ in the lifted diagram $\widetilde{\Gamma}(\mathbb{K})$ as being located in a larger $x_1$-level. Both lifts $H_{a_i}$, $V_{b_j}$ live on the core of the union $\mathcal{L}(\widetilde{\Gamma}(\mathbb{K}))=\bigsqcup_{k=1}^{m}M_k$. By the construction of each middle part $M_k$, the Legendrian thickening of $H_{a_i}$ is still \textit{crossing over} that of $V_{b_j}$ \textit{along} the $x_1$-direction because Legendrian thickenings (near such over-under crossings) are performed along the $x_2$-direction only, and so they are performed in the $x_1$-level where $H_{a_i}$, $V_{b_j}$ live. Moreover, the components of $\mathbb{L}'$ can be rounded (in their small neighborhoods) in such a way that the resulting $L_i''$'s are disjoint, and so we obtain a smooth Legendrian $2$-link $\mathbb{L}''=\bigsqcup_{i=1}^{r} L_{i}''$ (with the equatorial link $\mathbb{K}''$) which is Legendrian isotopic to the original link $\mathbb{L}$. In particular, the linking matrices of $\mathbb{L}$, $\mathbb{L}''$ are the same, and indeed, they coinside with those of $\mathbb{K}$ and $\mathbb{K}''$. This is one of the motivations of the autors in defining relative $2$-links and their generalized square bridge diagrams.
	\end{remark}
	
	
	\section{The Algorithm} \label{sec:algorithm}
	
	\subsection{Proof of Theorem \ref{thm:main thm}}
	Start with the generalized square bridge diagram $\Delta(\mathbb{L})$ of a given admissible Legendrian $2$-link $\mathbb{L}$ in $(\mathbb{R}^5,\xi_0^5)$ which consists of the square bridge diagram $\Gamma(\mathbb{K})$ of the isotropic equatorial link $\mathbb{K}$ which is Legendrian in $(\mathbb{R}^3,\xi_0^3)$. As in the previous section, one can construct a Legendrian diamond $\mathcal{D}_k=\mathcal{D}^{'}_k\cup M_k\cup \mathcal{D}^{''}_k$ for each square region $R_k$ in $\Gamma(\mathbb{K})$. In what follows, we will construct a $4$-dimensional \textit{ribbon} $\mathcal{R}_{\mathcal{D}_k}$ for each such diamond. Then by iteratively joining a ribbon $\mathcal{R}_{\mathcal{D}_k}$ to the union of the ribbons constructed in the previous step, we will reach the union of all ribbons so that at each step the union will be realized as a page of a compatible open book on $\mathbb{S}^5_{st}$ such that it contains the corresponding page of a simultanously evolving compatible open book on $\mathbb{S}^3_{st}$. We will achieve this by making use of relative positive stabilizations of compatible relative open book decompositions. In order to construct each $\mathcal{R}_{\mathcal{D}_k}$, we will construct a suitable ribbon for each Legendrian plates of $\mathcal{D}_k$ so that their union is the required ribbon $\mathcal{R}_{\mathcal{D}_k}$. \\
	
	For the middle part $M_k$ of the diamond $\mathcal{D}_k$, first consider its $S_{{I}_{k}}$ part which consists of $16$ Legendrian triangles. We construct $16$ ribbons $\mathcal{R}_{S_{{I}_{k}}}$ ($4$-dimensional compact hypersurface in $(\mathbb{R}^5,\xi_0^5)$ with corners) by following the $4$-dimensional contact planes along these $16$ triangular plates. More precisely, $$\mathcal{R}_{S_{{I}_{k}}}\doteq\{(p,v) \; : \; p \in S_{{I}_{k}}, \;\; v \in \xi_0^5|_p,\;\; \Vert v\Vert \leq \epsilon \}\subset \xi_0^5|_{S_{{I}_{k}}}.$$
	
	For each rectangular parts $S_{{H}_{i}}$, $S_{{V}_{j}}$ (there are 4 of them in total) of $M_k$, we consider two types of ribbons in $(\mathbb{R}^5,\xi_0^5)$ given by
	\begin{eqnarray*}
		\mathcal{R}_{S_{{H}_{i}}} &\doteq& \{z=-y_1+a_i, \quad 1-\epsilon  \leq x_1 \leq 1+\epsilon ,\quad -1 \leq x_2 \leq 1\}\\
		\mathcal{R}_{S_{{V}_{j}}} &\doteq& \{z=y_1+b_j, \quad -1-\epsilon  \leq x_1 \leq -1+\epsilon ,\quad -1 \leq x_2 \leq 1\}
	\end{eqnarray*}
	which are topologically the $2$-dimensional thickenings of $S_{{H}_{i}}$, $S_{{V}_{j}}$. One can check that the boundaries of these $20$ ribbons $\mathcal{R}_{S_{{I}_{k}}}, \mathcal{R}_{S_{{H}_{i}}}, \mathcal{R}_{S_{{V}_{j}}}$ meet along their $3$-dimensional common parts which are the $2$-dimensional thickenings of the common interval parts of the boundaries of $S_{{I}_{k}}, S_{{H}_{i}}, S_{{V}_{j}}$. (Recall the picture of $M_k$ in Figure \ref{Fig:Middle part} and \ref{Fig: Middle part 2}.) Hence, we construct the ribbon $\mathcal{R}_{M_k}$ of $M_k$ by gluing $20$ ribbons $\mathcal{R}_{S_{{I}_{k}}}, \mathcal{R}_{S_{{H}_{i}}}, \mathcal{R}_{S_{{V}_{j}}}$ along the common parts of their boundaries, i.e., $$\mathcal{R}_{M_k}=\mathcal{R}_{S_{{H}_{i}}} \cup \mathcal{R}_{S_{{I}_{K}}} \cup \mathcal{R}_{S_{{V}_{j}}}.$$ 
	By construction, $\mathcal{R}_{M_k}$ deformation retracts onto $M_k$, and the contact planes are tangent to $\mathcal{R}_{M_k}$ only along its Legendrian core $M_k$. Hence, $\mathcal{R}_{M_k}$ is the ribbon of $M_k$ in the contact geometric sense used by Giroux in \cite{Gi}. Moreover, observe that the ribbon $\mathcal{R}_{M_k}$ properly contains the $2$-dimensional ribbon $\mathcal{R}_{\gamma_k}$  of the Legendrian unknot $\gamma_k$ (the isotropic core of $M_k$)  which is defined by taking the union $\mathcal{R}_{I_k} \subset \mathcal{R}_{S_{{I}_{k}}}$ of the $4$ narrow bands defined by following the $2$-dimensional contact planes along the $4$ arcs $I_k$ in $(\mathbb{R}^3,\xi_0^3)$ and the following Legendrian planes:
	\begin{eqnarray*}
		\mathcal{R}_{{H}_{a_i}} &\doteq&\{z=-y_1+a_i, \quad 1-\epsilon  \leq x_1 \leq 1+\epsilon ,\quad x_2=y_2=0\} \subset \mathcal{R}_{S_{{H}_{i}}}\\
		\mathcal{R}_{{V}_{b_j}} &\doteq& \{z=y_1+b_j, \quad -1-\epsilon  \leq x_1 \leq -1+\epsilon ,\quad x_2=y_2=0\} \subset \mathcal{R}_{S_{{V}_{j}}}.
	\end{eqnarray*}	
	Next we construct the ribbons of the plates in the Legendrian disks $\mathcal{D}^{'}_k$, $\mathcal{D}^{''}_k$. For the triangular plates $T_{a_i}^{1,1}, T_{a_i}^{1,-1}, T_{b_j}^{-1,1}, T_{b_j}^{-1,-1}$, we define corresponding ribbons as 
	\begin{eqnarray*}
		\mathcal{R}_{T_{a_i}^{1,1}}&\doteq & \{z=-y_1-y_2+a_{i},\quad 1-\epsilon \leq x_1 \leq 1+ \epsilon,\quad 1-\delta \leq x_2 \leq 1+\delta\}\\
		\mathcal{R}_{T_{a_i}^{1,-1}}&\doteq & \{z=-y_1+y_2+a_{i},\quad 1-\epsilon \leq x_1 \leq 1+ \epsilon,\quad -1-\delta \leq x_2 \leq -1+\delta\}\\
		\mathcal{R}_{T_{b_j}^{-1,1}}&\doteq & \{z=y_1-y_2+b_{j},\quad -1-\epsilon \leq x_1 \leq -1+ \epsilon,\quad 1-\delta \leq x_2 \leq 1+\delta\}\\
		\mathcal{R}_{T_{b_j}^{-1,-1}}&\doteq & \{z=y_1+y_2+b_{j}, \quad -1-\epsilon \leq x_1 \leq -1+ \epsilon,\quad -1-\delta \leq x_2 \leq -1+\delta\}.
	\end{eqnarray*}
	for sufficiently small $\epsilon, \delta > 0$. The ribbons $\mathcal{R}_{T_{a_i}^{1,1}}, \mathcal{R}_{T_{a_i}^{1,-1}}, \mathcal{R}_{T_{b_j}^{-1,1}}, \mathcal{R}_{T_{b_j}^{-1,-1}}$ are disjoint and they deformation retract onto $T_{a_i}^{1,1}, T_{a_i}^{1,-1}, T_{b_j}^{-1,1}, T_{b_j}^{-1,-1}$, respectively.\\
	
	Now we consider $4$ other ribbon types $\mathcal{R}_{P_{\pm 1}}$, $\mathcal{R}_{P^{'}_{\pm 1}}$ whose cores are the Legendrian rectangular plates $P_{\pm 1}$, $P^{'}_{\pm 1}$ in $\mathcal{D}^{'}_k$ and $\mathcal{D}^{''}_k$. Their tangent spaces at a point $q$ are given by
	\begin{eqnarray*}
		T_q\mathcal{R}_{P_{\pm 1}}&=&\left\langle\frac{\partial}{\partial y_2}\pm \frac{\partial}{\partial z},\quad  \frac{\partial}{\partial x_1},\quad \frac{\partial}{\partial x_2},\quad \frac{\partial}{\partial z}-x_1\frac{\partial}{\partial y_1}\right\rangle\\
		T_q\mathcal{R}_{P^{'}_{\pm 1}}&=&\left\langle\frac{\partial}{\partial y_1}\pm \frac{\partial}{\partial y_2},\quad  \frac{\partial}{\partial x_1}- \frac{\partial}{\partial x_2},\quad \frac{\partial}{\partial z}-x_1\frac{\partial}{\partial y_1}, \quad \frac{\partial}{\partial z}-x_2\frac{\partial}{\partial y_2}\right\rangle.
	\end{eqnarray*}
	As the final ribbon types we consider $\mathcal{R}_{P_{t}}$ and $\mathcal{R}_{P_{b}}$ whose cores are the Legendrian top and bottom plates in $\mathcal{D}^{'}_k$ and $\mathcal{D}^{''}_k$. The tangent spaces of them at a point $q$ are given by
	$$T_q\mathcal{R}_{P_{t}}=T_q\mathcal{R}_{P_{b}}=\left\langle\frac{\partial}{\partial x_1},\quad \frac{\partial}{\partial x_2},\quad \frac{\partial}{\partial z}-x_1\frac{\partial}{\partial y_1},\quad \frac{\partial}{\partial z}-x_2\frac{\partial}{\partial y_2}\right\rangle.$$
	
	\noindent By taking appropriate $\epsilon$ and $\delta$, the ribbons $\mathcal{R}_{T_{a_i}^{1,1}}, \mathcal{R}_{T_{a_i}^{1,-1}}, \mathcal{R}_{T_{b_j}^{-1,1}}, \mathcal{R}_{T_{b_j}^{-1,-1}}, \mathcal{R}_{P_{\pm 1}}$, $\mathcal{R}_{P^{'}_{\pm 1}}, \mathcal{R}_{P_{t}},\mathcal{R}_{P_{b}}$ meet along their $3$-dimensional common boundary parts which are the $2$-dimensional thickenings of the common interval parts of the boundaries of their cores $T_{a_i}^{1,1}, T_{a_i}^{1,-1}, T_{b_j}^{-1,1}, T_{b_j}^{-1,-1}$, $P_{\pm 1}$, $P^{'}_{\pm 1}, P_t,P_b$. So, as before, gluing them along their common boundary regions gives two disjoint ribbons (both are topologically $D^4$) $\mathcal{R}_{\mathcal{D}^{'}_k}$ and $\mathcal{R}_{\mathcal{D}^{''}_k}$ whose cores are $\mathcal{D}^{'}_k$, $\mathcal{D}^{''}_k$, respectively. \\
	
	Finally, by taking the union of $\mathcal{R}_{M_k}$, $\mathcal{R}_{\mathcal{D}^{'}_k}$ and $\mathcal{R}_{\mathcal{D}^{''}_k}$ (along the thickening of the regions where $M_k, \mathcal{D}^{'}_k, \mathcal{D}^{''}_k$ meets), we construct the ribbon $$\mathcal{R}_{\mathcal{D}_k}=\mathcal{R}_{\mathcal{D}^{'}_k}\cup\mathcal{R}_{M_k}\cup\mathcal{R}_{\mathcal{D}^{''}_k}$$ whose core is the Legendrian diamond $\mathcal{D}_k$. By construction the $4$-dimensional ribbon $\mathcal{R}_{\mathcal{D}_k}$ properly contains the $2$-dimensional ribbon $R_{\gamma_k}$ which is diffeomorphic to $D^{*}\mathbb{S}^1$. Indeed,
	
	\begin{lemma} \label{lem:ribbon_is_cotangent_bdle}
		After rounding the corners, the ribbon $\mathcal{R}_{\mathcal{D}_k}$ is diffeomorphic to $D^{*}\mathbb{S}^2$.
	\end{lemma}
	
	\begin{proof}
		First, we observe that the tangent bundle of the rectangular plates $S_{{H}_{i}}$ and $S_{{V}_{j}}$ in the middle part $M_k$ are generated by $\frac{\partial}{\partial y_1}\pm \frac{\partial}{\partial z}$ and $\frac{\partial}{\partial x_2}$. The ribbons $\mathcal{R}_{S_{{H}_{i}}}$ and $\mathcal{R}_{S_{{V}_{j}}}$ are, respectively, obtained by thickening of $S_{{H}_{i}}$ and $S_{{V}_{j}}$  along the vectors $\frac{\partial}{\partial x_1}$ and $\frac{\partial}{\partial y_2}-x_2\frac{\partial}{\partial z}$. Also, we note that the tangent bundle of the triangular plates $S_{{I}_{k}}$ in $M_k$ is generated by  $\frac{\partial}{\partial x_1}$ and $\frac{\partial}{\partial x_2}$, and the ribbon $\mathcal{R}_{S_{{I}_{k}}}$ is the thickening of ${I}_{k}$ along the two vectors $\frac{\partial}{\partial y_1}-x_1\frac{\partial}{\partial z}$ and $\frac{\partial}{\partial y_2}-x_2\frac{\partial}{\partial z}$.\\
		
		Now, consider the isotropic unknot $\gamma_k$ in the middle part $M_k$. The tangent bundle of ${H}_{i}$ and ${V}_{j}$ in $\gamma_k$ are generated by  $\frac{\partial}{\partial y_1}\pm \frac{\partial}{\partial z}$. The $2$-dimesnsional ribbon $\mathcal{R}_{{H}_{i}}$ and $\mathcal{R}_{{V}_{j}}$ are obtained by thickening of ${H}_{i}$ and ${V}_{j}$ in the direction of $\frac{\partial}{\partial x_1}$. Also, the tangent bundle of ${I}_{k}$ is generated by $\frac{\partial}{\partial x_1}$. The $2$-dimesnsional ribbon $\mathcal{R}_{{I}_{k}}$ is the thickening of ${I}_{k}$ along $\frac{\partial}{\partial y_1}-x_1\frac{\partial}{\partial z}$. One can observe that $\mathcal{R}_{\gamma_k}$ topologically   positive Hopf band. So, the $4$-dimensional ribbon $\mathcal{R}_{M_k}$  contains  the $2$-dimensional ribbon $\mathcal{R}_{\gamma_k}$ which is topologically $D^{*}\mathbb{S}^1$, and the thickening directions of $\mathcal{R}_{\gamma_k}$ are  the cofiber direction of $\mathcal{R}_{\gamma_k}$ and the manifold direction of $\mathcal{R}_{M_k}$. Since the ribbon $\mathcal{R}_{\mathcal{D}_k}$ is  obtained by attaching two ribbons $\mathcal{R}_{\mathcal{D}^{'}_k}$ and $\mathcal{R}_{\mathcal{D}^{''}_k}$, which are topologically $D^{*}D^2$, to $\mathcal{R}_{M_k}$, $\mathcal{R}_{\mathcal{D}_k}$ is topologically $D^{*}\mathbb{S}^2$.
	\end{proof}
	
	\begin{remark} \label{rem:conformally_equivalent}
		In fact, we can say more here: By straigthforward computations one can check that $d\alpha_0^5|_{\mathcal{D}_k}$ is non-degenerate on the tangent space $T_q\mathcal{R}_{\mathcal{D}_k}$ for all $q\in \mathcal{R}_{\mathcal{D}_k}$, and so the pair $(\mathcal{R}_{\mathcal{D}_k},d\alpha_0^5|_{\mathcal{D}_k})$ is a symplectic manifolds. Indeed, one can realize it as the unit disk cotangent bundle of $\mathcal{D}_k\cong\mathbb{S}^2$. More precisely, by construction we know that $\mathcal{R}_{\gamma_k}$ is a positive Hopf band equipped with the symplectic structure $d\alpha_0^3|_{\gamma_k}=d\alpha_0^5|_{\gamma_k}$,  and so $(\mathcal{R}_{\gamma_k},d\alpha_0^3|_{\mathcal{D}_k})$ is the unit disk cotangent bundle of $\gamma_k\cong\mathbb{S}^1$ which is symplectically and properly embedded in $(\mathcal{R}_{\mathcal{D}_k},d\alpha_0^5|_{\mathcal{D}_k})$. Also by construction the zero section $\gamma_k$ is properly embedded in $\mathcal{D}_k$, and the contraction deforming $\mathcal{R}_{\mathcal{D}_k}$ onto its core $\mathcal{D}_k$ is, simultanously, deforming $\mathcal{R}_{\gamma_k}$ onto the core $\gamma_k$. Hence, $(\mathcal{R}_{\mathcal{D}_k},d\alpha_0^5|_{\mathcal{D}_k})$ must be the unit disk cotangent bundle of the $2$-sphere $\mathcal{D}_k$.
	\end{remark}	
	
	By Lemma \ref{lem:ribbon_is_cotangent_bdle} and Remark \ref{rem:conformally_equivalent}, we conclude that, for each Legendrian diamond $\mathcal{D}_k$, one can obtain a compatible open book $(\mathcal{R}_{\mathcal{D}_k}, \tau_{\mathcal{D}_k})\cong(D^{*}\mathbb{S}^2, \tau_{\mathbb{S}^2})$ on $\mathbb{S}^5_{st}$ where we identify (after rounding corners) $\mathcal{D}_k$ with the zero section $\mathbb{S}^2$ in $D^{*}\mathbb{S}^2$, and $\tau_{\mathcal{D}_k}$ denotes the right-handed Dehn twist along $\mathcal{D}_k$. Indeed, simultaneously, we also obtain a compatible open book $(\mathcal{R}_{\gamma_k}, \tau_{\gamma_k})\cong(D^{*}\mathbb{S}^1, \tau_{\mathbb{S}^1})$ on $\mathbb{S}^3_{st}$ where we identify (after rounding corners) $\gamma_k$ with the zero section $\mathbb{S}^1$ in $D^{*}\mathbb{S}^1$, and $\tau_{\gamma_k}$ denotes the right-handed Dehn twist along $\gamma_k$ (which is the equatorial circle of $\mathcal{D}_k$). Thus, for each Legendrian diamond $\mathcal{D}_k$ we have, indeed, a compatible relative open book on the (admissible) relative contact pair $(\mathbb{S}^5_{st},\mathbb{S}^3_{st})$ given by $$(\mathcal{R}_{\mathcal{D}_k}, \mathcal{R}_{\gamma_k}, \tau_{\mathcal{D}_k}, \tau_{\gamma_k})\cong(D^*\mathbb{S}^2,D^*\mathbb{S}^1,\tau_{\mathbb{S}^2},\tau_{\mathbb{S}^1}).$$
	
	Next, starting from the square bridge diagram $\Delta(\mathbb{L})$, we will construct a compatible relative open book on  $(\mathbb{S}^5_{st},\mathbb{S}^3_{st})$:
	
	\begin{proposition} \label{prop:Construction_via_stabilizations}
		For a given admissible Legendrian $2$-link $\mathbb{L}$ with its equatorial $1$-link $\mathbb{K}$, suppose that there are $m$ square regions $R_1, R_2, ..., R_m$ in the square bridge diagram $\Delta(\mathbb{K})$, and for each $k\in\{1,2, ..., m\}$ let $\mathcal{D}_k$ be a smooth (rounded) Legendrian $2$-unknot with its smooth (rounded) isotropic equatorial $1$-unknot $\gamma_k$ constructed as above. Then there exists a compatible relative open book $(F_{\mathbb{S}^5},F_{\mathbb{S}^3},h_{\mathbb{S}^5},h_{\mathbb{S}^3})$ on the (admissible) relative contact pair $(\mathbb{S}^5_{st},\mathbb{S}^3_{st})$ such that the pages $F_{\mathbb{S}^5},F_{\mathbb{S}^3}$ are Weinstein, and the following are satisfied:
		\begin{enumerate} 
			\item[(i)] $F_{\mathbb{S}^3}$ is symplectically and properly embedded in $F_{\mathbb{S}^5}$. 
			\item[(ii)] The monodromies of $(F_{\mathbb{S}^5},h_{\mathbb{S}^5})$, $(F_{\mathbb{S}^3},h_{\mathbb{S}^3})$ are, respectively, given by $$h_{\mathbb{S}^5}=\tau_{\mathcal{D}_1}\circ \tau_{\mathcal{D}_2}\circ \cdots \circ \tau_{\mathcal{D}_m}, \quad h_{\mathbb{S}^3}=\tau_{\gamma_1}\circ \tau_{\gamma_2}\circ \cdots \circ \tau_{\gamma_m}.$$
		\end{enumerate} 
	\end{proposition} 
	
	\begin{proof}
		We will introduce $R_k$'s to the square bridge diagram $\Delta(\mathbb{K})$ (and so the corresponding octahedrons to the generalized square bridge diagram $\Delta(\mathbb{L})$) one by one (in the order of their indices), and each step we will obtain a new compatible relative open book on $(\mathbb{S}^5_{st},\mathbb{S}^3_{st})$ as the result of a positive relative stabilization of the compatible relative open book constructed in the previous step. As an initial step, if no squares is introduced yet, then we consider the trivial compatible relative open book  $\mathcal{OB}_0\doteq(D^4, D^2, id_{D^4}, id_{D^2})$ on $(\mathbb{S}^5_{st},\mathbb{S}^3_{st})$ given in Example \ref{ex:trivial_open_book}. When the first square $R_1$ is introduced, from Example \ref{ex:basic_stabilization} we obtain 
		$$\mathcal{OB}_1\doteq(\mathcal{R}_{\mathcal{D}_1}, \mathcal{R}_{\gamma_1}, \tau_{\mathcal{D}_1}, \tau_{\gamma_1})=\mathcal{S}^+[(D^4, D^2, id_{D^4}, id_{D^2});A_1, A_1\cap D^2]$$ where $A_1$ is a properly embedded Lagrangian $2$-disk in $D^4$ with a Legendrian boundary in $\partial D^4$ as described in the example. We note that $\mathcal{R}_{\mathcal{D}_1}=D^4 \cup \Sigma_1$ and $\mathcal{R}_{\gamma_1}=D^2 \cup \sigma_1$ where $\Sigma_1$ and $\sigma_1$ are the corresponding Weinstein handles of the stabilizations mentioned in Remark \ref{rem:alternative_def_of_stabilization}.
		
		\noindent Now introducing $R_2$ yields
		$$\mathcal{OB}_2\doteq(\mathcal{R}_{\mathcal{D}_1}\cup \mathcal{R}_{\mathcal{D}_2}, \mathcal{R}_{\gamma_1}\cup \mathcal{R}_{\gamma_2}, \tau_{\mathcal{D}_1}\circ\tau_{\mathcal{D}_2}, \tau_{\gamma_1}\circ\tau_{\gamma_2})\cong\mathcal{S}^+[(\mathcal{R}_{\mathcal{D}_1}, \mathcal{R}_{\gamma_1}, \tau_{\mathcal{D}_1}, \tau_{\gamma_1});A_2, A_2\cap \mathcal{R}_{\gamma_1}].$$ Here $\mathcal{R}_{\mathcal{D}_1}\cup \mathcal{R}_{\mathcal{D}_2}=\mathcal{R}_{\mathcal{D}_1}\cup \Sigma_2$ and $\mathcal{R}_{\gamma_1}\cup \mathcal{R}_{\gamma_2}=\mathcal{R}_{\gamma_1}\cup \sigma_2$ where $$\Sigma_2=\mathcal{R}_{\mathcal{D}_2}\setminus \mathcal{R}_{M_1\cap M_2},\quad \sigma_2=\mathcal{R}_{\gamma_2}\setminus \mathcal{R}_{\gamma_1\cap \gamma_2}$$ are the corresponding Weinstein handles of the stabilizations so that the properly embedded Lagrangian $2$-disk $A_2$ in the ribbon $\mathcal{R}_{M_1 \cap M_2}$ obtained by gluing (smoothy) the standard Lagrangian cylinder $S^1\times [0,1]= \mathcal{R}_{M_1 \cap M_2}\cap \mathcal{D}_2$ to the Lagrangian disk $M_1 \cap M_2$ so that $S^1 \times \{0\}=\partial(M_1 \cap M_2)$ and $S^1\times \{1\}=\partial A_2$ is the attaching cirle of $\Sigma_2$ (and so $\sigma_2$ is attached along the pair of points $\partial (A_2\cap \mathcal{R}_{M_1\cap M_2})=\partial A_2\cap \partial \mathcal{R}_{M_1\cap M_2}$). This will attach the (new) diamond $\mathcal{D}_{2}$ to the existent core $\mathcal{D}_{1}$, and they together define the core $\mathcal{D}_{1}\cup \mathcal{D}_{2}$ of the new page (see Figure \ref{Fig: attaching}).\\
		
		We remark that taking the union of the Lagrangian core disk (say $C_2$) of $\Sigma_2$ and $A_2=(M_1 \cap M_2) \cup_{S^1 \times \{0\}}S^1\times [0,1]$ gives the Lagrangian $2$-sphere $$\mathcal{D}_2=(M_1 \cap M_2) \cup_{S^1 \times \{0\}}S^1\times [0,1]\cup_{S^1 \times \{1\}} C_2$$ in the new symplectic (indeed, Weinstein) page $\mathcal{R}_{\mathcal{D}_1}\cup \mathcal{R}_{\mathcal{D}_2} \subset \mathbb{S}^5_{st}$ along which a new right-handed Dehn twist is adjoined to the previous monodromy. In particular, $$\mathcal{D}_2\cap \mathbb{S}^3_{st}=[(M_1 \cap M_2) \cup_{S^1 \times \{0\}}S^1\times [0,1]\cup_{S^1 \times \{1\}} C_2]\cap \mathbb{S}^3_{st}=\gamma_2$$ is the Lagrangian $1$-sphere in the new symplectic (Weinstein) page $\mathcal{R}_{\gamma_1}\cup \mathcal{R}_{\gamma_2} \subset \mathbb{S}^3_{st}$ along which a new right-handed Dehn twist is adjoined to the previous monodromy.\\
		
		\begin{figure}[H]
			\begin{center}
				\includegraphics[width=0.6\textwidth]{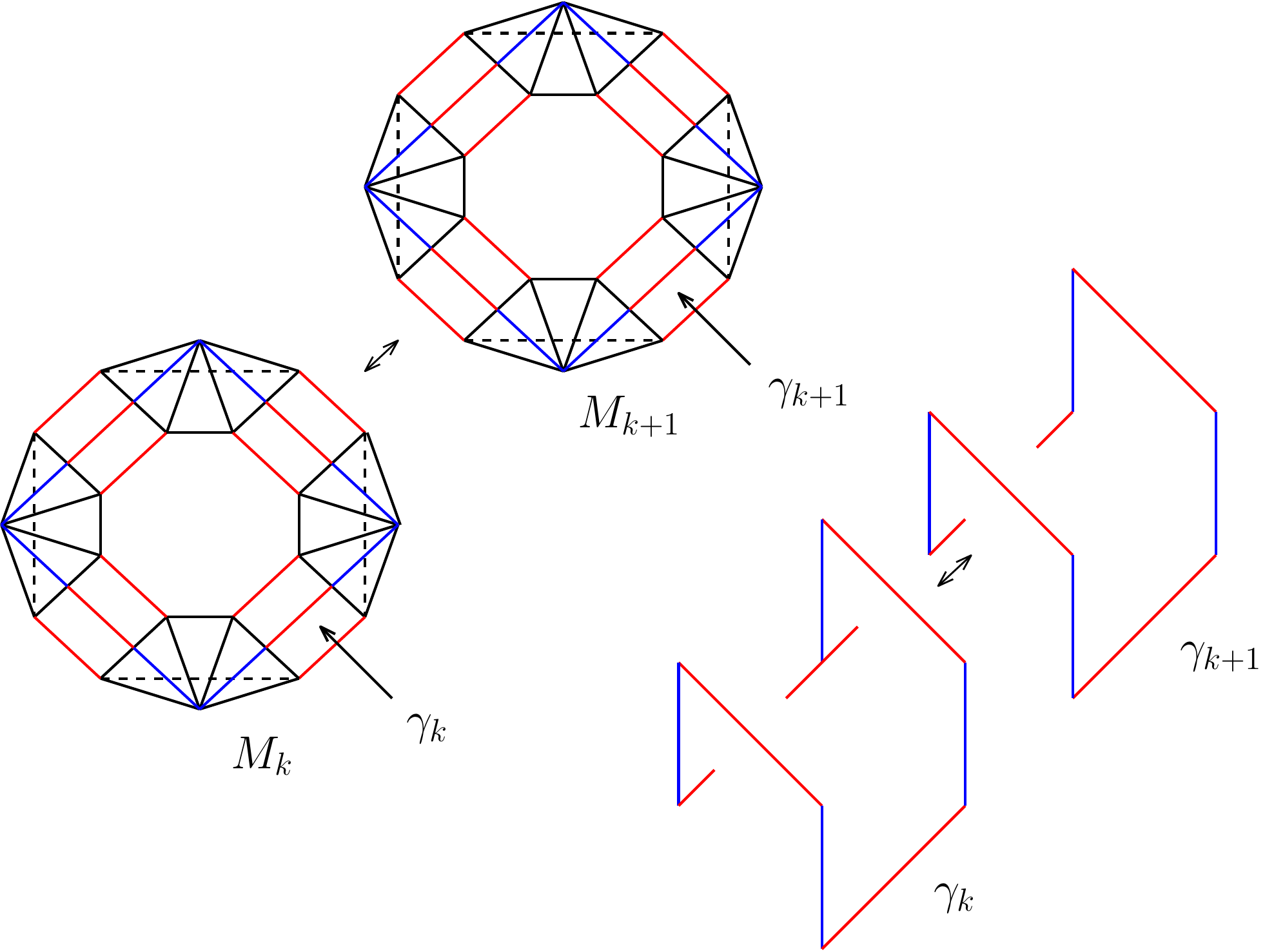}
				\caption{An attaching a new diamond $\mathcal{D}_{k+1}$ to an existent $\mathcal{D}_k$.}
				\label{Fig: attaching}
			\end{center}
		\end{figure}
		
		Iteratively we continue as follows: Assuming $R_k$ has been introduced, and so the corresponding open book $\mathcal{OB}_k\doteq(\bigsqcup_{i=1}^{k} \mathcal{R}_{\mathcal{D}_i}, \bigsqcup_{i=1}^{k}\mathcal{R}_{\gamma_i}, \tau_{\mathcal{D}_1}\circ\cdots \circ\tau_{\mathcal{D}_k}, \tau_{\gamma_1}\circ\cdots \circ\tau_{\gamma_k})$ has been constructed, introducing $R_{k+1}$ yields the compatible relative open book
		
		\begin{eqnarray*}
			\mathcal{OB}_{k+1}&\doteq &((\bigsqcup_{i=1}^{k} \mathcal{R}_{\mathcal{D}_i}) \cup\mathcal{R}_{\mathcal{D}_{k+1}}, (\bigsqcup_{i=1}^{k} \mathcal{R}_{\gamma_i})\cup\mathcal{R}_{\gamma_{k+1}}, (\tau_{\mathcal{D}_1}\circ \cdots \circ\tau_{\mathcal{D}_k})\circ\tau_{\mathcal{D}_{k+1}}, (\tau_{\gamma_1}\circ\cdots \circ\tau_{\gamma_k})\circ\tau_{\gamma_{k+1}})\\
			&=&\mathcal{S}^+[(\bigsqcup_{i=1}^{k} \mathcal{R}_{\mathcal{D}_i}, \bigsqcup_{i=1}^{k} \mathcal{R}_{\gamma_i}, \tau_{\mathcal{D}_1}\circ\cdots \circ\tau_{\mathcal{D}_k}, \tau_{\gamma_1}\circ\cdots \circ\tau_{\gamma_k});A_{k+1}, A_{k+1}\cap \bigsqcup_{i=1}^{k} \mathcal{R}_{\gamma_i}]
		\end{eqnarray*}
		
		As before we have $$(\bigsqcup_{i=1}^{k} \mathcal{R}_{\mathcal{D}_i}) \cup\mathcal{R}_{\mathcal{D}_{k+1}}=(\bigsqcup_{i=1}^{k} \mathcal{R}_{\mathcal{D}_i})\cup \Sigma_{k+1}, \quad (\bigsqcup_{i=1}^{k} \mathcal{R}_{\gamma_i})\cup\mathcal{R}_{\gamma_{k+1}}=(\bigsqcup_{i=1}^{k} \mathcal{R}_{\gamma_i})\cup \sigma_{k+1}$$ where $$\Sigma_{k+1}=\mathcal{R}_{\mathcal{D}_{k+1}}\setminus \mathcal{R}_{(\bigsqcup_{i=1}^{k}M_i)\cap M_{k+1}}, \quad \sigma_{k+1}=\mathcal{R}_{\gamma_{k+1}}\setminus \mathcal{R}_{(\bigsqcup_{i=1}^{k}\gamma_i)\cap \gamma_{k+1}}$$ are the corresponding Weinstein handles of the stabilizations so that the properly embedded Lagrangian $2$-disk $A_{k+1}$ in the ribbon $\mathcal{R}_{(\bigsqcup_{i=1}^{k}M_i)\cap M_{k+1}}$ obtained by gluing (smoothy) the standard Lagrangian cylinder $S^1\times [0,1]= \mathcal{R}_{(\bigsqcup_{i=1}^{k}M_i)\cap M_{k+1}}\cap \mathcal{D}_{k+1}$ to the Lagrangian disk $(\bigsqcup_{i=1}^{k}M_i) \cap M_{k+1}$ so that $S^1 \times \{0\}=\partial((\bigsqcup_{i=1}^{k}M_i)\cap M_{k+1})$ and $S^1\times \{1\}=\partial A_{k+1}$ is the attaching cirle of $\Sigma_{k+1}$, and so $\sigma_{k+1}$ is attached along the pair of points $$\partial (A_{k+1}\cap (\bigsqcup_{i=1}^{k}M_i)\cap M_{k+1})=\partial A_2\cap \partial \mathcal{R}_{(\bigsqcup_{i=1}^{k}M_i)\cap M_{k+1}}.$$ This will attach the (new) diamond $\mathcal{D}_{k+1}$ to the existent core $\bigsqcup_{i=1}^{k}\mathcal{D}_{i}$, and they together define the core $(\bigsqcup_{i=1}^{k}\mathcal{D}_{i})\cup \mathcal{D}_{k+1}$ of the new page (see Figure \ref{Fig: attachmiddle} where we take $m=4$).
		
		\begin{figure}[H]
			\begin{center}
				\includegraphics[width=0.32\textwidth]{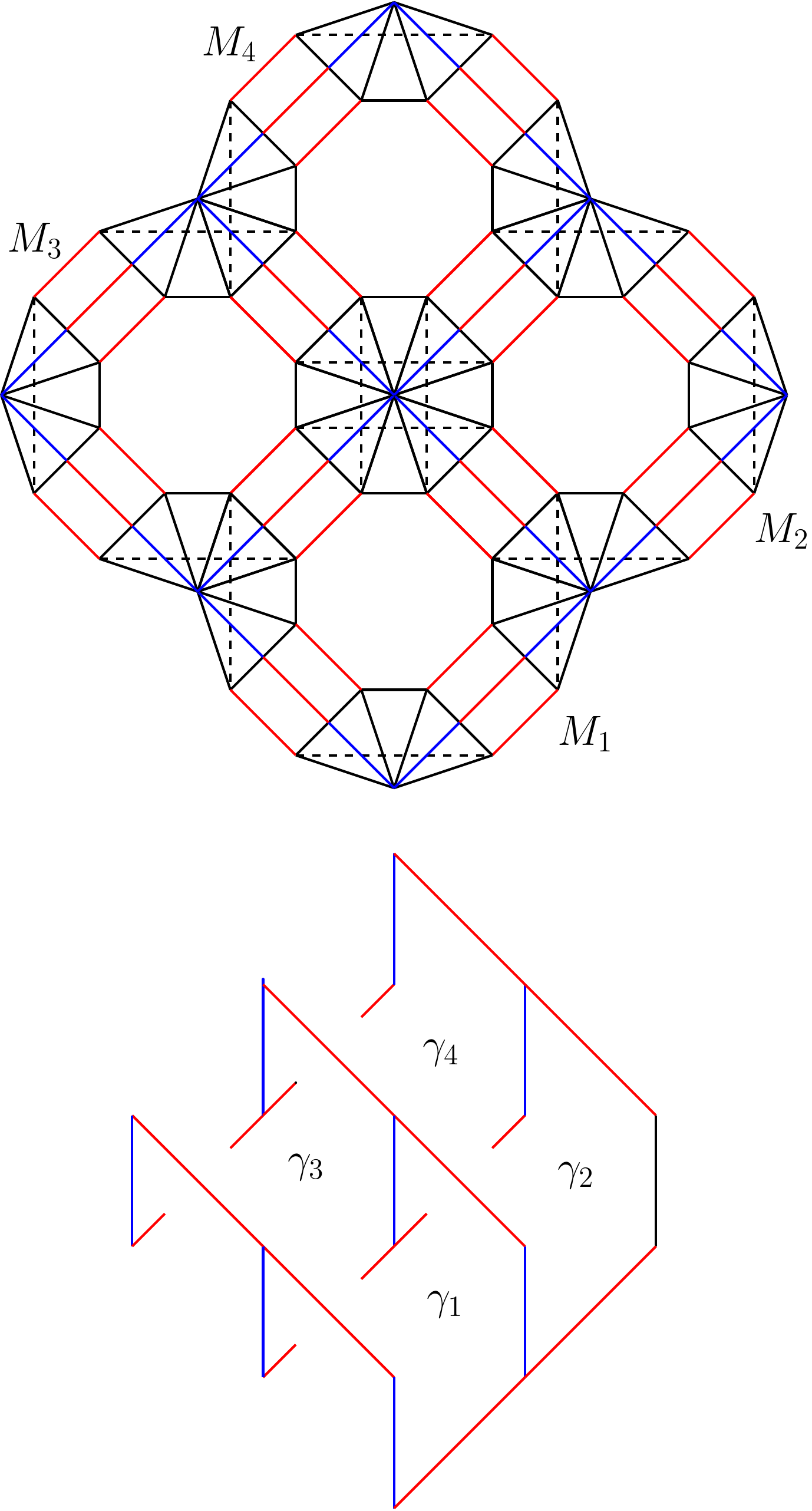}
				\caption{Attaching process along the middle parts $M_i$'s of $\mathcal{D}_{i}$'s.}
				\label{Fig: attachmiddle}
			\end{center}
		\end{figure}
		
		Similar to $k=2$ case, taking the union of the Lagrangian core disk $C_{k+1}$ of $\Sigma_{k+1}$ and $A_{k+1}=((\bigsqcup_{i=1}^{k}M_i)\cap M_{k+1}) \cup_{S^1 \times \{0\}}(S^1\times [0,1])$ gives the Lagrangian $2$-sphere $$\mathcal{D}_{k+1}=((\bigsqcup_{i=1}^{k}M_i)\cap M_{k+1}) \cup_{S^1 \times \{0\}}(S^1\times [0,1])\cup_{S^1 \times \{1\}} C_{k+1}$$ in the new symplectic (Weinstein) page $(\bigsqcup_{i=1}^{k} \mathcal{R}_{\mathcal{D}_i}) \cup\mathcal{R}_{\mathcal{D}_{k+1}} \subset \mathbb{S}^5_{st}$ along which a new right-handed Dehn twist is adjoined to the previous monodromy. In particular, $$\mathcal{D}_{k+1}\cap \mathbb{S}^3_{st}=[((\bigsqcup_{i=1}^{k}M_i)\cap M_{k+1}) \cup_{S^1 \times \{0\}}(S^1\times [0,1])\cup_{S^1 \times \{1\}} C_{k+1}]\cap \mathbb{S}^3_{st}=\gamma_{k+1}$$ is the Lagrangian $1$-sphere in the new symplectic (Weinstein) page $(\bigsqcup_{i=1}^{k} \mathcal{R}_{\gamma_i})\cup\mathcal{R}_{\gamma_{k+1}} \subset \mathbb{S}^3_{st}$ along which a new right-handed Dehn twist is adjoined to the previous monodromy.\\
		
		Hence, by taking $k=m$, i.e., after introducing all $R_k$'s, we obtain a relative open book
		
		\begin{eqnarray*}
			(F_{\mathbb{S}^5},F_{\mathbb{S}^3},h_{\mathbb{S}^5},h_{\mathbb{S}^3})&\doteq &\mathcal{OB}_{m}\\
			&=& \mathcal{S}^+[\cdots [\mathcal{S}^+[(D^4, D^2, id_{D^4}, id_{D^2});A_1, A_1\cap D^2];\cdots ];A_m, A_m\cap \bigsqcup_{i=1}^{m-1} \mathcal{R}_{\gamma_i}]
		\end{eqnarray*}
		which is compatible with the admissible relative contact pair $(\mathbb{S}^5_{st},\mathbb{S}^3_{st})$ by Theorem \ref{thm:Relative_Stab_Compatible_Open Book}. Moreover, by construction, the pages $F_{\mathbb{S}^3}, F_{\mathbb{S}^5}$ are Weinstein (as being constructed from Weinstein handles), and the conditions (i), (ii) are satisfied as required.
	\end{proof}
	
	Returning back to the proof of Theorem \ref{thm:main thm}, starting from the generalized square bridge diagram $\Delta(\mathbb{L})$, consisting of $m$ octahedrons, of the given link, we know, by Proposition \ref{prop:Construction_via_stabilizations}, that there exists a compatible relative open book $(F_{\mathbb{S}^5},F_{\mathbb{S}^3},h_{\mathbb{S}^5},h_{\mathbb{S}^3})$ on $(\mathbb{S}^5_{st},\mathbb{S}^3_{st})$ with Weinstein pages and satisfying (v) in the statement. Moreover, for any component $L$ of $\,\mathbb{L}$ and its equatorial knot $K\subset \mathbb{K}$, we know, by Proposition \ref{Prop:reconstruction}, that we can reconstruct $L$ as a piecewise smooth Legendrian $2$-sphere $L' \subset \mathbb{S}^5_{st}$ and by rounding the corners of $L'$, one can obtain a smooth Legendrian $2$-sphere which is Legendrian isotopic to $L$. Indeed, by construction, $L'$ (and so its smoothening) is contained in the core $\bigsqcup_{i=1}^{m}\mathcal{D}_{i}$ of the page $F_{\mathbb{S}^5}=\bigsqcup_{i=1}^{m}\mathcal{R}_{\mathcal{D}_{i}}$. In particular, the piecewise smooth Legendrian equatorial $1$-sphere $K' \subset \mathbb{S}^3_{st}$ (and its smoothening) is contained in the core $\bigsqcup_{i=1}^{m}\gamma_{i}$ of the page $F_{\mathbb{S}^3}=\bigsqcup_{i=1}^{m}\mathcal{R}_{\gamma_{i}}$. We depict this reconstruction for the standard Legendrian $2$-unknot in Figure \ref{Fig:Cuspitaledge}.
	
	\begin{figure}[H]
		\begin{center}
			\includegraphics[width=0.5\textwidth]{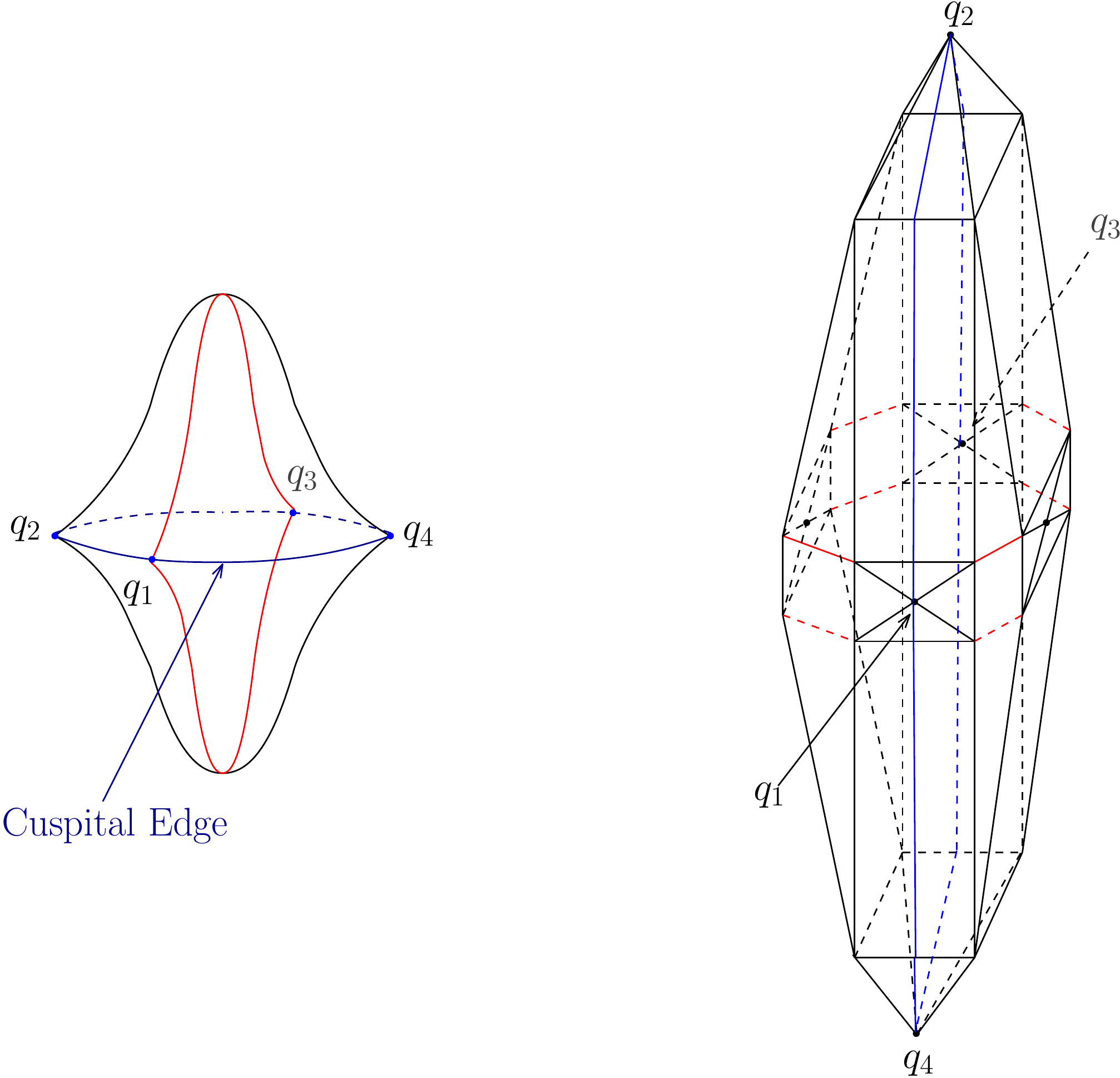}
			\caption{On the left, the front projection of the (admissible) standard Legendrian $2$-unknot $L$, and on the right, the corresponding piecewise smooth Legendrian $2$-unknot $L'$ (reconstruction of $L$).}
			\label{Fig:Cuspitaledge}
		\end{center}
	\end{figure}
	
	So far, we have proved the statements (i) and (iii) (considering $L$ as the rounded $L'$). Since the reconstruction of $K$ (rounded $K'$) in the page $F_{\mathbb{S}^3}$ is based on the square bridge diagram $\Delta(K)$, the statement (ii) follows directly from the same lines in \cite{Ar}. In order to prove (iv), we observe that the arguments used in proving (ii) can be generalized to high dimensions in a natural way. More precisely, since $L$ is contained in the core $\bigsqcup_{i=1}^{m}\mathcal{D}_{i}$, its ribbon $\mathcal{R}_L$ is contained in the page $F_{\mathbb{S}^5}$ (the ribbon of $\bigsqcup_{i=1}^{m}\mathcal{D}_{i}$). Note that reconstruction process does not change the Legendrian isotopy type of $L$, so the contact (Thurston-Bennequin) framing of $L$ is still the same as what it was at the beginning. Moreover, consider tubular neighborhood $N(L)$ of $L$ in the page $F_{\mathbb{S}^5}$. Being a codimension zero submanifold of the ribbon $F_{\mathbb{S}^5}$, $N(L)$ is the ribbon of $L$. By definition, the contact framing of $L$ is the one coming from the ribbon $N(L)$. Therefore, the contact framing and the $N(L)$-framing (page framing) of $L$ are the same. Since $N(L) \subset F_{\mathbb{S}^5}$, the framing which $L$ gets from the page $F_{\mathbb{S}^5}$ is the same as the contact framing of $L$.
	This completes the proof of Theorem \ref{thm:main thm}.
	
	\qed	
	
	
	\subsection{Admissible Relative Contact Surgery and Proof of Theorem \ref{thm:contact surgery}} \label{sec:admissible-surgery}
	
	In this section, we will describe a(n) \textit{(admissible) relative contact $r-$surgery}  along a(n) (admissible) relative Legendrian $2$-knot $L$ in $(\mathbb{S}^{5}_{st},\mathbb{S}^{3}_{st})$ which gives us a new relative contact pair, and then we will consider a special type of it which consists of a contact (Legendrian)  $(\pm 1)$-surgery along $L$ in  $\mathbb{S}^{5}_{st}$ and a simultaneously performed contact (Legendrian) $(\pm 1)$-surgery along the equatorial $1$-knot $K$ of $L$ in $\mathbb{S}^{3}_{st}$, and prove Theorem \ref{thm:contact surgery}.\\
	
	A contact surgery along a Legendrian $1$-knot in a contact $3$-manifold is a topolological surgery where the surgery coefficient measured with respect to contact framing: Let $K$ be an Legendrian $1$-knot in a contact manifold $(M^{3},\eta)$ with a trivial neigborhood $N(K)\cong \mathbb{D}^2 \times \mathbb{S}^1$ with convex boundary. The boundary $\partial N(K)$ has two homology generators represented by a \textit{meridian} $\mu$ and a \textit{longitude} $\lambda$ where $\lambda$ is a parallel copy of $K$ obtained by pushing $K$ using its contact framing. Choose a trivilization $\phi: \mathbb{D}^2 \times \mathbb{S}^1 \rightarrow N(K)$ such that $\phi_*(\partial \mathbb{D}^2 \times \{pt\})=\mu$ and $\phi_*(\{pt\}\times \mathbb{S}^1)=\lambda$ where $\partial \mathbb{D}^2 \times \{pt\}$ and $\{pt\}\times \mathbb{S}^1$ are the generators of $H_1(\mathbb{S}^1\times \mathbb{S}^1)$. A \textit{contact} $r=(\frac{p}{q})-$\textit{surgery along} $K$ gives a new (not necessarily unique) contact manifold  $\displaystyle (\mathbb{D}^2 \times \mathbb{S}^1) \cup_{\phi_r} (M^3\setminus N(K)) $ where $ \phi_r: \partial( \mathbb{D}^2 \times \mathbb{S}^1) \rightarrow \partial(M^3\setminus N(K))$ is a gluing map such that $(\phi_r)_*(\partial \mathbb{D}^2 \times \{pt\})=p\mu+q\lambda$. In the case of $r=\pm1$, a contact surgery produces a unique (up to isotopy) contact structure on the surgered (resulting) $3$-manifold. For details see \cite{Gm} and \cite{OS}.\\
	
	In higher dimensions, one could also consider contact surgeries along isotropic and coisotropic spheres (see \cite{CE}). However, a contact surgery along a Legendrian $n$-knot is defined in a similar way: It is a topological surgery where the surgery coefficient is determined relative to contact framing.  Let $\mathcal{S}$ be an Legendrian $n$-knot in a contact manifold $(Y^{2n+1},\xi)$ with a trivial neigborhood $N(\mathcal{S})\cong \mathbb{D}^{n+1} \times \mathbb{S}^n$ with convex boundary. The boundary $\partial N(\mathcal{S})$ has two homology generators represented by a \textit{meridian} $m$ and a \textit{longitude} $\ell$ which is a parallel copy of $\mathcal{S}$ obtained by pushing $\mathcal{S}$ using its contact framing.  Choose a trivilization $\Phi: \mathbb{D}^{n+1} \times \mathbb{S}^n \rightarrow N(\mathcal{S})$ such that $\Phi_*(\partial \mathbb{D}^{n+1} \times \{pt\})=m$ and $\Phi_*(\{pt\}\times \mathbb{S}^n)=\ell$ where $\partial \mathbb{D}^{n+1} \times \{pt\}$ and $\{pt\}\times \mathbb{S}^n$ are the generators of $H_n(\mathbb{S}^n\times \mathbb{S}^n)$. A \textit{contact} $r=(\frac{p}{q})-$\textit{surgery along} $\mathcal{S}$ gives a new (not necessarily unique) contact manifold  $\displaystyle (\mathbb{D}^{n+1} \times \mathbb{S}^n) \cup_{\Phi_r} (Y^{2n+1}\setminus N(\mathcal{S})) $ where $ \Phi_r: \partial( \mathbb{D}^{n+1} \times \mathbb{S}^n) \rightarrow \partial(Y^5\setminus N(\mathcal{S}))$ is a gluing map such that $(\Phi_r)_*(\partial \mathbb{D}^{n+1} \times \{pt\})=p\mu+q\lambda$. As before, in the case of $r=\pm1$, a contact surgery produces a unique (up to isotopy) contact structure on the surgered (resulting) $(2n+1)$-manifold.\\
	
	Returning back to the dimension three and five, we now introduce a \textit{relative $r-$surgery}  along a relative $2$-knot. Although it can be defined in a purely topological way, we give the definition in a contact setting:
	
	\begin{definition} \label{def:relative_contact_surgery}
		Let $L$ be a relative Legendrian $2$-knot in $\mathbb{S}^{5}_{st}$. By Legendrian neighborhood theorem, $L$ has a neighborhood $(N(L), \xi^5_{st}|_{N(L)})$ contactomorphic to a neighborhood of the zero section in the $1$-jet space $(T^{*}L\times \mathbb{R}, \textrm{Ker}(dz-\sum_{i=1}^{2}p_idq_i))$. The contact $r-$surgery along $L$ gives a new (not unique) contact $5$-manifold $(Y^{5},\xi)$ obtained by removing $N(L)$ from $\mathbb{S}^{5}_{st}$, and gluing $\mathbb{D}^3 \times \mathbb{S}^2$ back via a diffeomorphism inducing a map $\Phi_{r}$ on the boundary (as above). Since $L$ is a relative knot, it has an equatorial circle $K$ which is Legendrian in $\mathbb{S}^{3}_{st}$ and its trivial neighborhood $(N(K), \xi_{st}^3|_{N(K)})$ is contactomorphic to a tubular neighborhood of the zero section in $\displaystyle (T^{*}K\times \mathbb{R}, \textrm{Ker}(dz-p_1dq_1))$ which is contactly embedded in $(T^{*}L\times \mathbb{R}, \textrm{Ker}(dz-\sum_{i=1}^{2}p_idq_i))$. Since $K$ is the isotropic equator of $L$, we have $T^{*}K\times \mathbb{R} \subset T^{*}L\times \mathbb{R}$, and indeed, the neighborhood  $N(K)\cong K\times D^{2}$ is a contact submanifold of $N(L)\cong L\times D^{3}$ where $\{a\}\times D^{2} \subset \{a\}\times D^{3}$ for any point $a \in K$. So, when we remove $N(L)$ from $\mathbb{S}^{5}_{st}$, we simultaneously remove $N(K)$ from $\mathbb{S}^{3}_{st}$. Therefore, by gluing $\mathbb{D}^{3} \times \mathbb{S}^2$, and so $\mathbb{D}^{2} \times \mathbb{S}^1(\subset \mathbb{D}^{3} \times \mathbb{S}^2)$, back via diffeomorphisms $\Phi_r: \partial(\mathbb{D}^{3} \times \mathbb{S}^2) \rightarrow \partial(\mathbb{S}^{5}_{st}\setminus N(L))$ and its restriction $\Phi_r|_{\partial (\mathbb{D}^{2} \times \mathbb{S}^1)}: \partial( \mathbb{D}^{2} \times \mathbb{S}^1) \rightarrow \partial(\mathbb{S}^{3}_{st}\setminus N(K))$, we obtain the surgered manifolds $$(Y^5,\xi):=(\mathbb{S}^{5}_{st}\setminus {N}(L)) \cup (\mathbb{D}^{3} \times \mathbb{S}^2), \quad M^3:=(\mathbb{S}^{3}_{st}\setminus N(K)) \cup (\mathbb{D}^{2} \times \mathbb{S}^1).$$ Note that we obtain a relative (smooth) topological pair $(Y, M)$. Indeed, one can check (see Lemma \ref{lem:Relative_contact_r_surgery}) that the $2$-plane distribution $\eta:=\xi \cap TM$ defines a contact structure on $M$ (so $(M,\eta)$ is a contact submanifold of $(Y,\xi)$), and $(M,\eta)$ is the result of a contact $r-$surgery performed simultaneously along $K$ in $\mathbb{S}^{3}_{st}$. In particular, we have obtained a relative contact pair $(Y_{\xi}, M_{\eta})$ which is said to be obtained by a \textit{relative contact $r$-surgery along a relative Legendrian $2$-knot} $L$ in $\mathbb{S}^{5}_{st}$.
	\end{definition}
	
	\begin{lemma} \label{lem:Relative_contact_r_surgery}
		In the setting of Definition \ref{def:relative_contact_surgery}, the $2$-plane distribution $\eta=\xi \cap TM$ is a contact structure on $M$. Moreover, $(M,\eta)$ is the result of a contact $r-$surgery along a Legendrian $1$-knot $K$ (the equator of $L$) in $\mathbb{S}^{3}_{st}$.
	\end{lemma}
	
	\begin{proof} First we note that the contact embedding of $N(K)\cong K\times D^{2}$ into $N(L)\cong L\times D^{3}$
		is also a proper embedding, i.e., $$\partial N(K)\cong K\times S^1 \subset L\times S^2\cong \partial N(L), \quad \textrm{with} \quad \{a\}\times S^1 \subset \{a\}\times S^2,\quad  \forall a \in K,$$ and the characteristic foliation on the convex boundary $\partial N(K) \subset \mathbb{S}^{3}_{st}$ is the restriction of the one on $\partial N(L) \subset \mathbb{S}^{5}_{st}$. On the other hand, the contact structure on $\mathbb{D}^3\times \mathbb{S}^2$ (the one chosen to perform the contact $r$-surgery along $L$ resulting $(Y,\xi)$) produces the characteristic foliation on the convex boundary $\mathbb{S}^2 \times \mathbb{S}^2$. This foliation restricts to the one on the convex boundary $\mathbb{S}^1 \times \mathbb{S}^1$ of the properly and contactly embedded submanifold $\mathbb{D}^2\times \mathbb{S}^1$ of $\mathbb{D}^3\times \mathbb{S}^2$. Moreover, when $\mathbb{D}^3\times \mathbb{S}^2$ is glued to $\mathbb{S}^{5}_{st}\setminus {N}(L)$ via $\Phi_r:\mathbb{S}^2 \times \mathbb{S}^2 \to \partial (\mathbb{S}^{5}_{st}\setminus {N}(L))$, the characteristic foliations on $\mathbb{S}^2 \times \mathbb{S}^2, \partial (\mathbb{S}^{5}_{st}\setminus {N}(L))$ are identified with each other so that a contact structure $\xi$ is obtained on the surgered manifold $Y$. In particular, the restriction diffeomorphism $\Phi_r|_{\mathbb{S}^1 \times \mathbb{S}^1}:\mathbb{S}^1 \times \mathbb{S}^1 \to \partial (\mathbb{S}^{3}_{st}\setminus {N}(K))$ glues the contact submanifolds $\mathbb{D}^2 \times \mathbb{S}^1 \subset \mathbb{D}^3\times \mathbb{S}^2$ and $\mathbb{S}^{3}_{st}\setminus {N}(K)\subset \mathbb{S}^{5}_{st}\setminus {N}(L)$ together so that the characteristic sub-foliations on $\mathbb{S}^1 \times \mathbb{S}^1$ and $\partial (\mathbb{S}^{3}_{st}\setminus {N}(K))$ are identified with each other. This means that when a contact surgery performed along $L$, a contact structure $\eta=\xi \cap TM$ is,  simultaneously, constructed on the surgered manifold $M$ as claimed. \\ 
		
		Now, the homology group  $H_2(\partial N(L))$ has two generators represented by a meridian $m$ and a longitude $\ell$ which is a parallel copy of $L$ on $\partial N(L)$ obtained by pushing $L$ using its contact framing. Choose a trivilization $$\Phi: \mathbb{D}^3\times \mathbb{S}^2\rightarrow  N(L)$$ such that $\Phi_*(\partial \mathbb{D}^3\times\{pt\})=m$ and $\Phi_*(\{pt\}\times \mathbb{S}^2)=\ell$ where $\partial \mathbb{D}^3\times\{pt\}$ and $\{pt\}\times \mathbb{S}^2$ are the generators of $H_2(\mathbb{S}^2\times \mathbb{S}^2)$. If $p$ and $q$ are coprime integers such that $r=\frac{p}{q}$, then the result of the contact $r-$surgery to $\mathbb{S}^{5}_{st}$ is the manifold $$(Y,\xi):=(\mathbb{S}^{5}_{st}\setminus N(L)) \cup_{\Phi_r} (\mathbb{D}^{3} \times \mathbb{S}^{2}) $$ where $(\Phi_r)_*(\partial \mathbb{D}^3\times\{pt\})=pm+q\ell$. As explained above simultaneously a surgery is being performed along the equatorial Legendrian circle $K$ in $\mathbb{S}^{3}_{st}$: Replacing $N(L)$ with $\mathbb{D}^3 \times\mathbb{S}^2$ corresponds to replacing $N(K)$ with $\mathbb{D}^2 \times\mathbb{S}^1$. For the homology group $H_1(\partial N(K))\cong \mathbb{Z}\oplus\mathbb{Z}$, consider the following two generators: Let $\mu \subset m$ be the meridian $\mu=m \cap \partial N(K)$, and $\lambda \subset \ell$ be the longitude $\lambda=\ell \cap \partial N(K)$. Note that $\mu$ is an equatorial circle of $m$ as $N(K)$ is properly embedded in $N(L)$, also that $\lambda$ is an equatorial circle of $\ell$ as $N(K)$ is contactly embedded in $N(L)$, so in particular, it is a parallel copy of $K$ produced by the contact framing of $K$ in $\mathbb{S}^{3}_{st}$. Now the trivilization $\Phi$ restricts to a trivilization $$\phi\doteq\Phi|_{\mathbb{D}^2\times \mathbb{S}^1}=\mathbb{D}^2\times \mathbb{S}^1 \rightarrow N(K)$$ such that $\phi_*(\partial \mathbb{D}^2\times\{pt\})=\mu$ and $\phi_*(\{pt\}\times \mathbb{S}^1)=\lambda$ where $\partial \mathbb{D}^3\times\{pt\}$ and $\{pt\}\times \mathbb{S}^2$ are the generators of $H_1(\mathbb{S}^1\times \mathbb{S}^1)$. Finally, observe that the gluing via the restriction map $\Phi_r|_{\partial (\mathbb{D}^{2} \times \mathbb{S}^1)}$ (resulting the contact submanifold $(M,\eta)$) is equivalent to gluing via the map $$\phi_r: \partial( \mathbb{D}^{2} \times \mathbb{S}^1) \rightarrow \partial(\mathbb{S}^{3}_{st}\setminus N(K))\quad \textrm{with}\quad  \phi_*(\partial \mathbb{D}^2\times\{pt\})=p\mu+q\lambda.$$  (Since the $2$-sphere $\Phi_r(\partial \mathbb{D}^3\times\{pt\})$ travels $p$-times around the meridian $m$ and $q$-times around the longitude $\ell$, so does its equator $\Phi_r(\partial \mathbb{D}^2\times\{pt\})$, i.e., it travels $p$-times around the meridian $\mu$ and $q$-times around the longitude $\lambda$.) Hence, the simultaneous surgery along $K$ constructing the contact submanifold $(M,\eta=\xi \cap TM)$ is also a contact $r-$surgery.  
	\end{proof}

	\begin{proof}[\textbf{Proof of Theorem \ref{thm:contact surgery}}]
		
		Suppose that $(Y^5_{\xi}, M^3_{\eta})$ is a given admissible relative contact pair obtained by an admissible relative contact $(\pm 1)$-surgery along an admissible Legendrian relative $2$-link $\mathbb{L}=\bigsqcup_{i=1}^{r} L_{i}$ in $\mathbb{S}^5_{st}$ whose equatorial isotropic $1$-link is $\mathbb{K}=\mathbb{L}\cap \mathbb{S}^3_{st}=\bigsqcup_{i=1}^{r} {K}_{i}$. We know, by Theorem \ref{thm:main thm}, that there exists a compatible relative open book $(F_{\mathbb{S}^5},F_{\mathbb{S}^3},h_{\mathbb{S}^5},h_{\mathbb{S}^3})$ on the relative contact pair $(\mathbb{S}^5,\mathbb{S}^3)$ with the Weinstein pages $F_{\mathbb{S}^5},F_{\mathbb{S}^3}$ and the monodromies 
		$$h_{\mathbb{S}^5}=\tau_{\mathcal{D}_1}\circ \tau_{\mathcal{D}_2}\circ\cdots \circ \tau_{\mathcal{D}_m}, \quad h_{\mathbb{S}^3}=\tau_{\gamma_1}\circ \tau_{\gamma_2}\circ \cdots \circ \tau_{\gamma_m}$$ (as constructed in the proof of the theorem) such that 
		each reconstructed and smoothened $K_i$ lies on $F_{\mathbb{S}^3}$ and the page framing of $K_i$ is equal to its contact framing in $\mathbb{S}^3_{st}$, and similarly, each reconstructed and smoothened $L_i$ lies on $F_{\mathbb{S}^5}$ and the page framing of $L_i$ is equal to its contact framing in $\mathbb{S}^5_{st}$. From know on, all reconstructed links and their reconstructed components are also assumed to be smoothened, and by abuse of notation, we'll denote them by the same letters.\\
		
		By construction, all components of the reconstructed $2$-link $\mathbb{L}$ sits on a single page, say $F^0_{\mathbb{S}^5}$, as embedded Legendrian $2$-unknots (of $\mathbb{S}^5_{st}$) in such a way that all components of the reconstructed equatorial $1$-link $\mathbb{K}$ lie on a page, say $F^0_{\mathbb{S}^3}$ (a Weinstein submanifold of $F^0_{\mathbb{S}^5}$) as embedded Legendrian $1$-unknots (of $\mathbb{S}^3_{st}$). Being Legendrian on a symplectic page of a compatible open book, all $L_i$'s (resp. $K_i$'s) are Lagrangian in $F^0_{\mathbb{S}^5}$ (resp. in $F^0_{\mathbb{S}^3}$). Recall that, by construction, the components of the reconstructed $\mathbb{K}$ are still mutually disjoint in $F^0_{\mathbb{S}^3}$ and the linking matrix of the reconstructed $\mathbb{K}$ is the same as the linking matrices of the original $\mathbb{K}$ and $\mathbb{L}$. However, the components of reconstructed $\mathbb{L}$ have some overlaping regions. In order to get rid of such regions, we will realize each reconstructed component $L_i$ on a different page, call $F^i_{\mathbb{S}^5}$, of the open book $(F_{\mathbb{S}^5},h_{\mathbb{S}^5})$ in such a way that each corresponding equatorial component $K_i$ will be contained on a (different) page $F^i_{\mathbb{S}^3}$ of the the open book $(F_{\mathbb{S}^3},h_{\mathbb{S}^3})$ which is a Weinstein subdomain of $F^i_{\mathbb{S}^5}$.\\
		
		First, note that since all reconstructed components of $\mathbb{L}$ are contained in the Lagrangian core of the page $F^0_{\mathbb{S}^5}$, they do not intersect the binding, so they miss a page of $(F_{\mathbb{S}^5},h_{\mathbb{S}^5})$, and hence we may assume that they are contained in a closed neighborhood $N_{\mathbb{S}^5}(F^0_{\mathbb{S}^5})$ of $F^0_{\mathbb{S}^5}$ in $\mathbb{S}^5$. This neighborhood can be contactly identified with $(F \times [-1,1], \textrm{Ker}(\alpha_0^5+dt))$ where $t$ is the coordinate on $[-1,1]$, $F= F^0_{\mathbb{S}^5}$ is equipped with the symplectic form $d\alpha_0^5$, and we identify $F^0_{\mathbb{S}^5}$ with the page $F \times \{0\}$. Thus, we have a contactomorphism (contact identification)
		$$\Theta: (N_{\mathbb{S}^5}(F^0_{\mathbb{S}^5}),\textrm{Ker}(\alpha_0^5|_{N_{\mathbb{S}^5}(F^0_{\mathbb{S}^5})}) \longrightarrow (F \times [-1,1], \textrm{Ker}(\alpha_0^5+dt)).$$ The contactomorphism $\Theta$ identifies the pages (fibers) in $N_{\mathbb{S}^5}(F^0_{\mathbb{S}^5})$ with those in $F \times [-1,1]$, and the Reeb vector field of $N_{\mathbb{S}^5}(F^0_{\mathbb{S}^5})$ corresponds to the Reeb vector field $\partial / \partial t$ of $F \times [-1,1]$. \\
		
		Similarly, all reconstructed components of $\mathbb{K}$ are contained in the Lagrangian core of the page $F^0_{\mathbb{S}^3}$, and they miss a page of $(F_{\mathbb{S}^3},h_{\mathbb{S}^3})$, and so they are contained in a closed neighborhood $N_{\mathbb{S}^3}(F^0_{\mathbb{S}^3})$ of $F^0_{\mathbb{S}^3}$ in $\mathbb{S}^3$ which can be contactly identified with $(G \times [-1,1], \textrm{Ker}(\alpha_0^3+dt))$ where $G= F^0_{\mathbb{S}^3}$ is equipped with the symplectic form $d\alpha_0^3$, and we identify $F^0_{\mathbb{S}^3}$ with the page $G \times \{0\}$. In fact, $G$ is a properly embedded Weinstein subdomain of $F$, and the restriction of $\Theta$ gives a contactomorphism 
		$$\Theta|_{N_{\mathbb{S}^3}(F^0_{\mathbb{S}^3})}: (N_{\mathbb{S}^3}(F^0_{\mathbb{S}^3}),\textrm{Ker}(\alpha_0^3|_{N_{\mathbb{S}^3}(F^0_{\mathbb{S}^3})}) \longrightarrow (G \times [-1,1], \textrm{Ker}(\alpha_0^3+dt)).$$ Likewise the restricted $\Theta$ identifies the pages (fibers) in $N_{\mathbb{S}^3}(F^0_{\mathbb{S}^3})$ with those in $G \times [-1,1]$, and the Reeb vector field of $N_{\mathbb{S}^3}(F^0_{\mathbb{S}^3})$ corresponds to the Reeb vector field $\partial / \partial t$ of $G \times [-1,1]$. (Note that $(G \times [-1,1], \textrm{Ker}(\alpha_0^3+dt))$ is contactly and properly embedded in $ (F \times [-1,1], \textrm{Ker}(\alpha_0^5+dt))$ and their Reeb vector fields coincide, namely $\partial /\partial t$.)\\
		
		Fix a metric $d$ on $F$ (and so on $G$) and let $d_0$ be the standard metric on $[-1,1]$. Consider the pull-back metric (distance function) \;$\textrm{dist}\doteq \Theta^*(d \times d_0)$ \;on $N_{\mathbb{S}^5}(F^0_{\mathbb{S}^5})$ where  is the identification map. Let $\epsilon>0$ be a sufficiently small number such that $\epsilon<\textrm{dist}(K_i,K_j)$ for all $i\neq j$, and choose a partition $\{-\epsilon=\epsilon_1\lneqq\epsilon_2\lneqq \cdots \lneqq\epsilon_r=\epsilon\}$ of the interval $[-\epsilon,\epsilon]$.\\
		
		Next, for each $i=1,2, ... r$, consider the image $\Theta(L_i)$ which is Lagrangian in $(F\times \{0\}, d\alpha_0^5)$ and Legendrian in $(F \times [-1,1], \textrm{Ker}(\alpha_0^5+dt))$.  Using the flow of the Reeb vector field $\partial/\partial t$ (which is transverse to the pages) we push the $2$-unknot $\Theta(L_i)$ untill we obtain a new smooth $2$-unknot $U_i$ which is completely contained in the page $F\times \{\epsilon_i\}$. Note that since $\partial/\partial t$ is a contact vector field, $\Theta(L_i)$ and $U_i$ are Legendrian isotopic in $(F \times [-1,1], \textrm{Ker}(\alpha_0^5+dt))$. The restriction of this isotopy to the equator $\Theta(K_i)$ defines a Legendrian isotopy in $(G \times [-1,1], \textrm{Ker}(\alpha_0^3+dt))$ from $\Theta(K_i)$ to a new smooth Legendrian $1$-unknot $V_i \subset G\times \{\epsilon_i\}\subset F\times \{\epsilon_i\}$ which is, indeed, the isotropic equator of $U_i$. Observe that $U_i$'s (and so $V_i$'s) are mutually disjoint as they lie on different pages. Also by the choice of $\epsilon$, the new smooth $2$-link $\mathbb{U}=\bigsqcup_{i=1}^{r} U_{i}$ and its smooth equatorial $1$-link $\mathbb{V}=\bigsqcup_{i=1}^{r} V_{i}$ have a common linking matrix which is, indeed, equal to the common linking matrix of the originally given link $\mathbb{L}$ and its equatorial link $\mathbb{K}$.\\
		
		Now consider the pull-back link $\mathbb{L}'\doteq\Theta^{-1}(\mathbb{U})\subset N_{\mathbb{S}^5}(F^0_{\mathbb{S}^5})\subset \mathbb{S}^5_{st}$ and its equatorial link $\mathbb{K}'\doteq\Theta^{-1}(\mathbb{V})\subset N_{\mathbb{S}^3}(F^0_{\mathbb{S}^3})\subset \mathbb{S}^3_{st}$. Since $\Theta$ is a contactomorphism, the following are immediate:
		\begin{itemize}
			\item[(1)] Each component $L'_i\doteq \Theta^{-1}(U_i)$ of $\mathbb{L}'$ is Lagrangian on $F^{i}_{\mathbb{S}^5}\doteq \Theta^{-1}(F\times \{\i\})$, and Legendrian in $\mathbb{S}^5_{st}$, and Legendrian isotopic to original $L_i$. 
			\item[(2)] Each component $K'_i\doteq \Theta^{-1}(V_i)$ of $\mathbb{K}''$ is Lagrangian on $F^{i}_{\mathbb{S}^3}\doteq \Theta^{-1}(G\times \{i\})$, and Legendrian in $\mathbb{S}^3_{st}$, and Legendrian isotopic to original $K_i$. 
			\item[(3)] The $2$-link $\mathbb{L}'=\bigsqcup_{i=1}^{r} L'_{i}$ and its equatorial $1$-link $\mathbb{K}'=\bigsqcup_{i=1}^{r} K'_{i}$ have a common linking matrix which coincides with the common linking matrix of $\mathbb{L}=\bigsqcup_{i=1}^{r} L_{i}$ and $\mathbb{K}=\bigsqcup_{i=1}^{r} K_{i}$.\\
		\end{itemize}
		
		As a result, we conclude that the relative $2$-links $\mathbb{L}$, $\mathbb{L}'$ and their equatorial $1$-links $\mathbb{K}$, $\mathbb{K}'$ are Legendrian isotopic in $\mathbb{S}^5_{st}$ and $\mathbb{S}^3_{st}$, respectively. Therefore, the relative contact $\pm 1$-surgery along $\mathbb{L}'$ also results in the same (up to isotopy) relative contact pair $(Y^5_{\xi}, M^3_{\eta})$. On the other hand, each component of $\mathbb{L}'$ (resp. $\mathbb{K}'$) sits as a Lagrangian $2$-unknot (resp. Lagrangian $1$-unknot) on the pages of the compatible open book $(F_{\mathbb{S}^5},h_{\mathbb{S}^5})$ on $\mathbb{S}^5_{st}$ (resp. $(F_{\mathbb{S}^3},h_{\mathbb{S}^3})$ on $\mathbb{S}^3_{st}$). Therefore, by ``surgery on a page'' theorems (Theorem \ref{thm:Etnyresurgery} for dimension three, and see Theorem 4.6 in \cite{Koert2} for dimension five), one immediately gets a compatible relative open book on the relative contact pair $(Y^5_{\xi}, M^3_{\eta})$ given by $$(F_{\mathbb{S}^5},F_{\mathbb{S}^3},h_{\mathbb{S}^5}\circ \tau_{\mathbb{L}'},h_{\mathbb{S}^3}\circ \tau_{\mathbb{K}'})$$
		where $\tau_{\mathbb{L}'}$ (resp. $\tau_{\mathbb{K}'}$) denotes the product (composition) of the Dehn twists $\tau^{\mp 1}_{L'_i}$ (resp. $\tau^{\mp 1}_{K'_i}$) about the components $\mathbb{L}'$ (resp. $\mathbb{K}'$) along which the contact $(\pm 1)$-surgeries, resulting $Y^5_{\xi}$ (resp. $M^3_{\eta}$), are performed. 
	\end{proof}
	
	
	\section{Examples}
	\label{sec:examples}
	\begin{example}
		
Consider the admissible Legendrian $2$-link $\mathbb{L}=\bigsqcup_{i=1}^{2} L_{i}$  whose equatorial link $\mathbb{K}$ is the Hopf link such that its square bridge diagram $\Gamma(\mathbb{K})$ is the one given in Figure \ref{Fig:sbd of Hopf link}. The generalized square bridge diagram $\Delta(\mathbb{L})$ of $\mathbb{L}$ is given Figure \ref{Fig:Hopf link}. Using these diagrams, for each $k$, we construct the Legendrian unknot $\gamma_k$ in $(\mathbb{R},\xi_0^3)\subset \mathbb{S}^3_{st}$ and the middle part (Legendrian thickening) $M_k$ of $\gamma_k$ as in Figure \ref{Fig:example1-hopf-middlepart}.
		
		\begin{figure}[H]
			\begin{center}
				\includegraphics[width=0.5\textwidth]{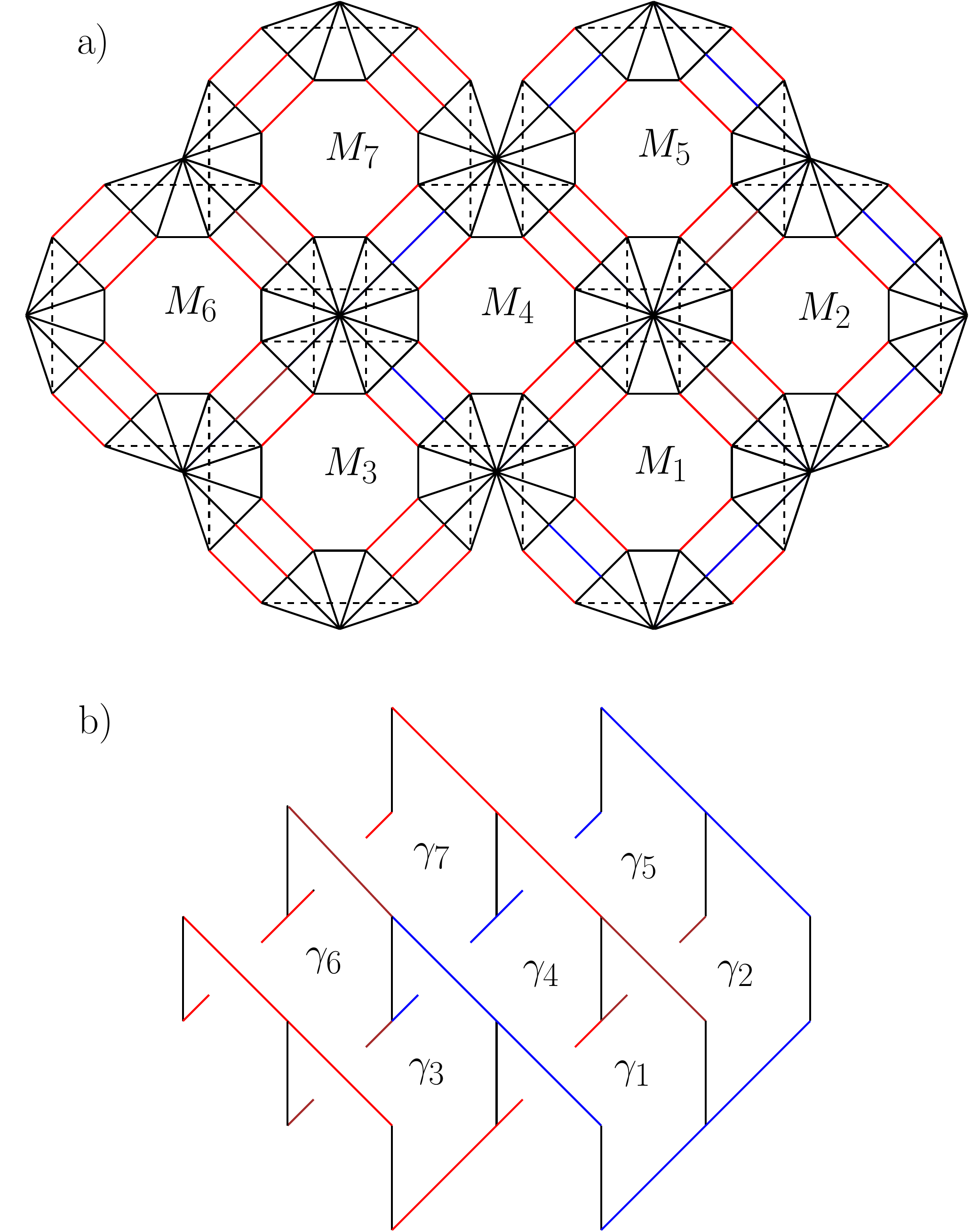}
				\caption{(a) The middle part of the admissible Legendrian $2$-link $\mathbb{L}$, (b) the equatorial link $\mathbb{K}$ of $\mathbb{L}$ (the Hopf link) contained in the core of the $2$-dimensional ribbon.}
				\label{Fig:example1-hopf-middlepart}
			\end{center}
		\end{figure}

Then one can obtain the piecewise smooth (Legendrian) reconstruction $\mathbb{L}'$ (with some overlapping regions) of $\mathbb{L}$ by first building a Legendrian diamond $\mathcal{D}_{k}$ for each square region $R_k$ in $\Gamma(\mathbb{K})$ and then patching their appropriate parts together as exaplained in Proposition \ref{Prop:reconstruction} (see Figure \ref{Fig:reconstruction hopf}). By construction, $\mathbb{L}'$ is contained in the union of $\mathcal{R}_{\mathcal{D}_k}$'s ($4$-dimensional ribbons), and its equator $\mathbb{K}'$ is contained in the union of $\mathcal{R}_{\gamma_k}$'s ($2$-dimensional ribbons).
 
		\begin{figure}[H]
			\begin{center}
				\includegraphics[width=0.5\textwidth]{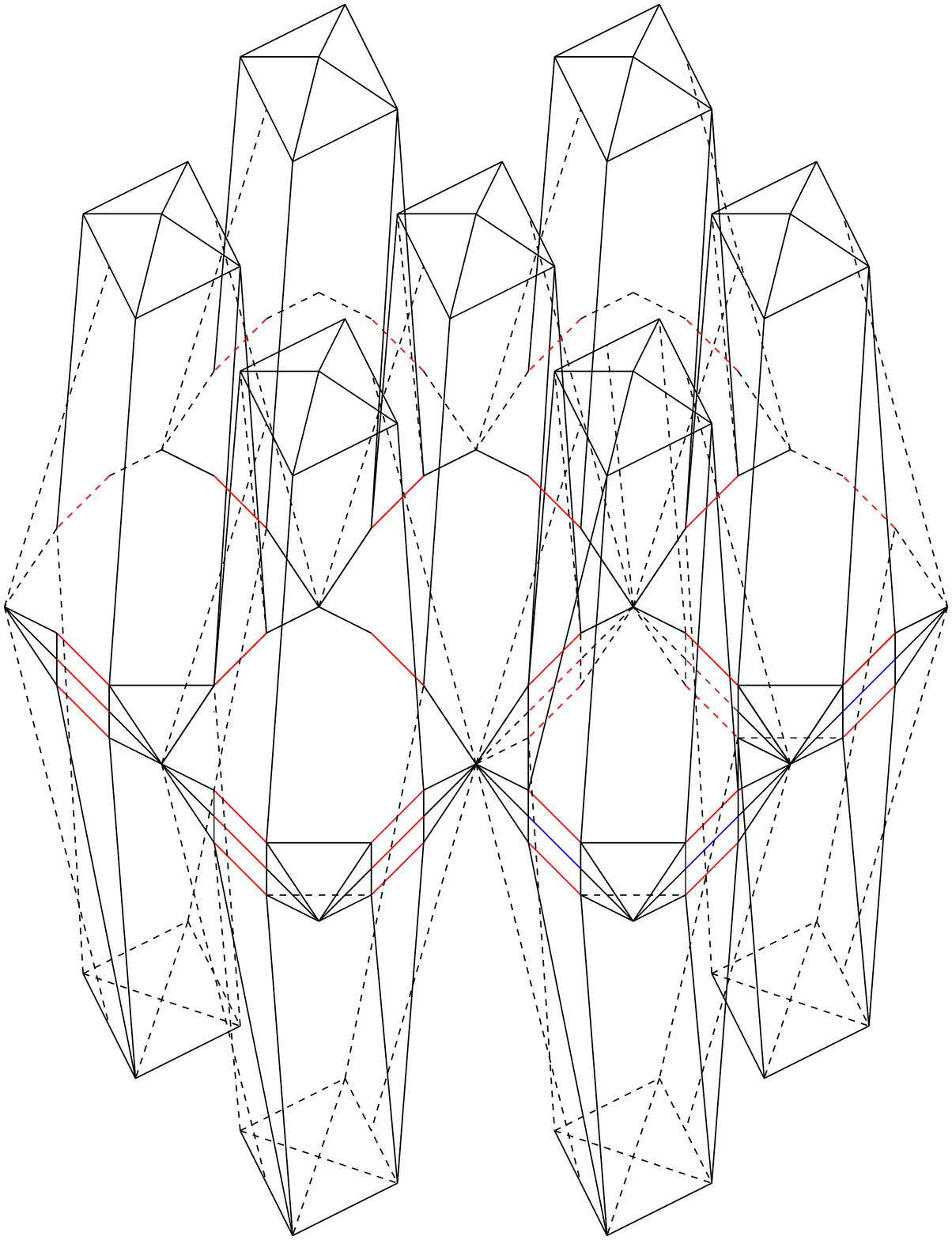}
				\caption{The piecewise smooth Legendrian $2$-link $\mathbb{L}'$ (with overlapping regions) and its equatorial link $\mathbb{K}'$ (a piecewise smooth Hopf link).}
				\label{Fig:reconstruction hopf}
			\end{center}
		\end{figure}
		
When each octahedron is introduced to the generalized square bridge diagram $\Delta(\mathbb{L})$ of $\mathbb{L}$ (step by step) for each square region $R_k$ in $\Gamma(\mathbb{K})$, a new compatible relative open book on  $(\mathbb{S}^5_{st},\mathbb{S}^3_{st})$ is obtained as the result of a positive relative stabilization of the compatible relative open book in the previous step.\\
		
More precisely, the trivial compatible relative open book $\mathcal{OB}_0\doteq(D^4, D^2, id_{D^4}, id_{D^2})$ on $(\mathbb{S}^5_{st},\mathbb{S}^3_{st})$ is considered as an initial step. When we reach the last step $k=7$ , we have the following compatible relative open book on $(\mathbb{S}^5_{st},\mathbb{S}^3_{st})$:
		\begin{eqnarray*}
			(F_{\mathbb{S}^5},F_{\mathbb{S}^3},h_{\mathbb{S}^5},h_{\mathbb{S}^3})&\doteq &\mathcal{OB}_{7}\\
			&=& \mathcal{S}^+[\cdots [\mathcal{S}^+[(D^4, D^2, id_{D^4}, id_{D^2});A_1, A_1\cap D^2];\cdots ];A_7, A_7\cap \bigsqcup_{i=1}^{6} \mathcal{R}_{\gamma_i}]
		\end{eqnarray*}
where $F_{\mathbb{S}^5}=\bigsqcup_{k=1}^{7} \mathcal{R}_{\mathcal{D}_k}$ and  $F_{\mathbb{S}^3}=\bigsqcup_{k=1}^{7} \mathcal{R}_{\gamma_k}$ are the corresponding Weinstein pages, the monodromies $h_{\mathbb{S}^5}=\tau_{\mathcal{D}_1}\circ\cdots \circ\tau_{\mathcal{D}_7}$ and $h_{\mathbb{S}^3}=\tau_{\gamma_1}\circ\cdots \circ\tau_{\gamma_7}$ are the product of right-handed Dehn twists along $\mathcal{D}_k$ and $\gamma_k$, respectively, and $A_{k}$ is the properly embedded Lagrangian $2$-disk  in the ribbon $\mathcal{R}_{(\bigsqcup_{i=1}^{k-1}M_i)\cap M_{k}}$ as described in Proposition \ref{prop:Construction_via_stabilizations} for $k=1, \cdots, 7$. Note that rounded (smoothened) $\mathbb{L}'$ and $\mathbb{K}'$ are smoothly embedded in the cores $\bigsqcup_{k=1}^{7} \mathcal{D}_k$ and $\bigsqcup_{k=1}^{7} \gamma_k$ of $F_{\mathbb{S}^5}=\bigsqcup_{k=1}^{7} \mathcal{R}_{\mathcal{D}_k}$ and  $F_{\mathbb{S}^3}=\bigsqcup_{k=1}^{7} \mathcal{R}_{\gamma_k}$, respectively.\\
		
Let $(Y_{\xi}^5, M_{\eta}^3)$ be any admissible relative contact pair obtained by an admissible relative contact $(\pm 1)$-surgery along the given admissible link $\mathbb{L}$ in $\mathbb{S}_{st}^5$. By Theorem \ref{thm:contact surgery}, there exists a compatible relative open book $$(\mathcal{OB}_Y, \mathcal{OB}_M)\doteq (F_{\mathbb{S}^5},F_{\mathbb{S}^3},h_{\mathbb{S}^5}\circ\tau_{\mathbb{L}'},h_{\mathbb{S}^3}\circ\tau_{\mathbb{K}'})$$ on $(Y_{\xi}^5, M_{\eta}^3)$ with Weinstein pages where $\tau_{\mathbb{L}'}$ and $\tau_{\mathbb{K}'}$ denote the product of the Dehn twists $\tau_{\mathbb{L}_i'}^{\mp 1}$ and $\tau_{\mathbb{K}_i'}^{\mp 1}$ as described in the previous section when the contact $(\pm 1)$-surgeries are applied.
		
	\end{example}
	\begin{example}
Let $L$ be the admissible standard loose Legendrian $2$-unknot  in $(\mathbb{R}^5, \xi_0^5)\subset \mathbb{S}^5_{\xi_{st}}$ which is obtained by taking $\mathbb{S}^1$-symmetry of the Legendrian $1$-unknot $K$ in $(\mathbb{R}^3, \xi_0^3)\subset \mathbb{S}^3_{\xi_{st}}$ with $tb(K)=-3$ as depicted in Figure \ref{Fig:Loose Legendrian}.
		\begin{figure}[H]
			\begin{center}
				\includegraphics[width=0.4\textwidth]{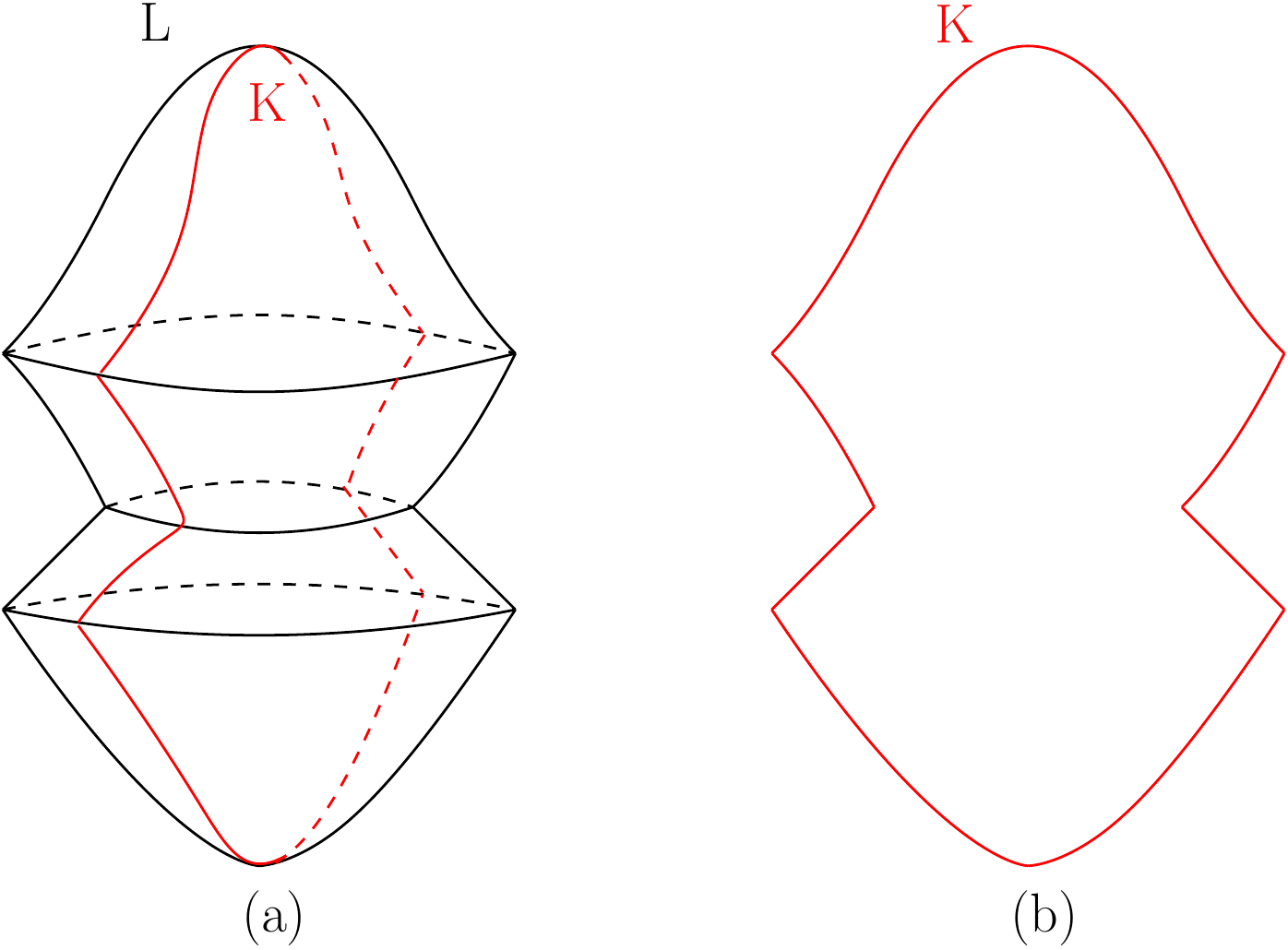}
				\caption{(a) The front of the admissible standard loose Legendrian $2$-unknot $L$, (b) the front of its isotropic equator $K$.}
				\label{Fig:Loose Legendrian}
			\end{center}
		\end{figure}
		
The generalized square bridge diagram $\Delta(L)$ is given in Figure \ref{Fig:loosegsbd}. By using this diagram, one can obtain the piecewise smooth representative $L'$ (as depicted in Figure \ref{Fig:loosediamond}) by first constructing a Legendrian diamond $\mathcal{D}_k$ for each square region $R_k$ in $\Gamma(K)$, and then patching their appropriate parts together as exaplained in Proposition \ref{Prop:reconstruction}.
		\begin{figure}[H]
			\begin{center}
				\includegraphics[width=0.20\textwidth]{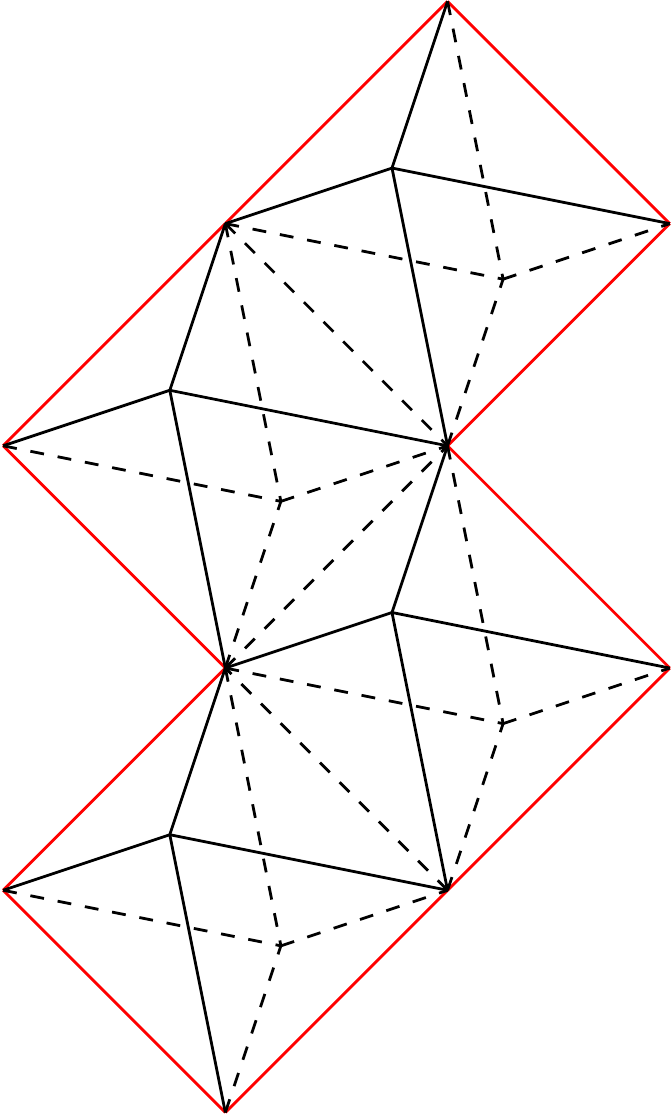}
				\caption{The generalized square bridge diagram of  the admissible standard loose Legendrian $2$-unknot $L$.}
				\label{Fig:loosegsbd}
			\end{center}
		\end{figure}
		\begin{figure}[H]
			\begin{center}
				\includegraphics[width=0.4\textwidth]{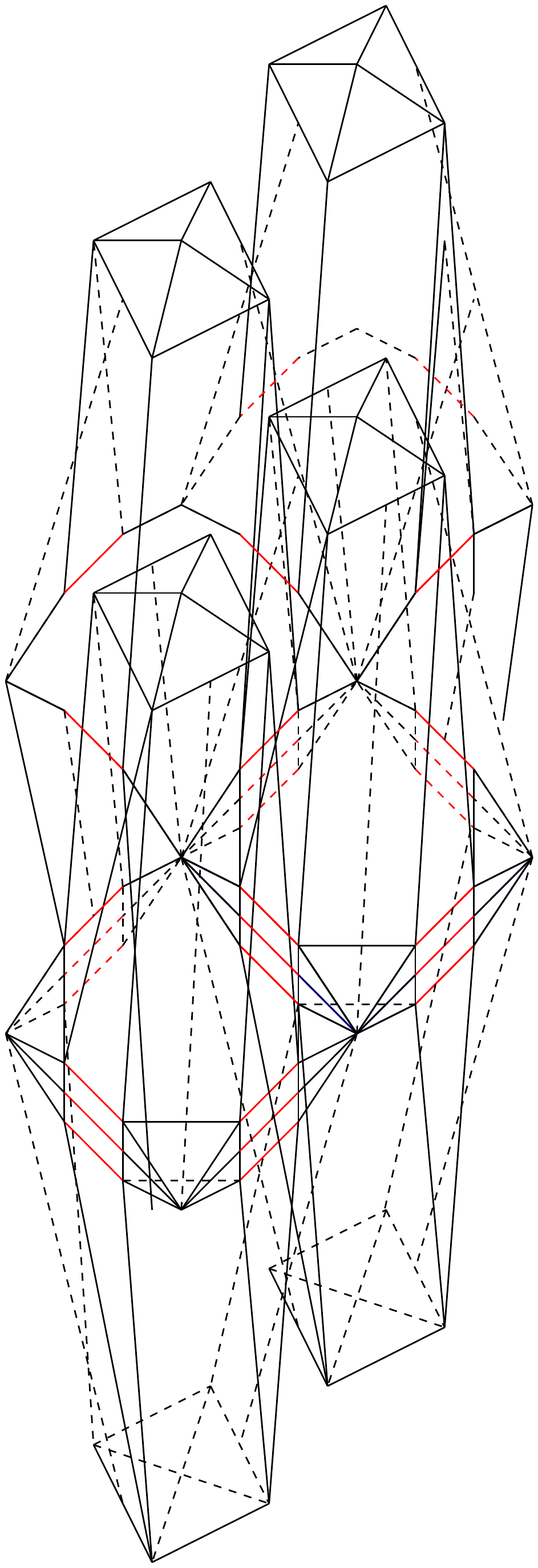}
				\caption{The piecewise smooth standard loose Legendrian $2$-unknot $L'$.}
				\label{Fig:loosediamond}
			\end{center}
		\end{figure}
		
As before, we consider the trivial compatible relative open book $\mathcal{OB}_0\doteq(D^4, D^2, id_{D^4}, id_{D^2})$ on $(\mathbb{S}^5_{st},\mathbb{S}^3_{st})$ when no squares is introduced. Then for each new octahedron, a new compatible relative open book on $(\mathbb{S}^5_{st},\mathbb{S}^3_{st})$ is obtained by positively stabilizing the compatible open book in the previous step. At the last step $k=4$, the following compatible open book
		\begin{eqnarray*}
			(F_{\mathbb{S}^5},F_{\mathbb{S}^3},h_{\mathbb{S}^5},h_{\mathbb{S}^3})&\doteq &\mathcal{OB}_{4}\\
			&=& \mathcal{S}^+[\cdots [\mathcal{S}^+[(D^4, D^2, id_{D^4}, id_{D^2});A_1, A_1\cap D^2];\cdots ];A_4, A_4\cap \bigsqcup_{i=1}^{3} \mathcal{R}_{\gamma_i}]
		\end{eqnarray*}
on $(\mathbb{S}^5_{st},\mathbb{S}^3_{st})$ is obtained where $F_{\mathbb{S}^5}=\bigsqcup_{k=1}^{4} \mathcal{R}_{\mathcal{D}_k}$ and  $F_{\mathbb{S}^3}=\bigsqcup_{k=1}^{4}\mathcal{R}_{\gamma_k}$ are the corresponding Weinstein pages, $h_{\mathbb{S}^5}=\tau_{\mathcal{D}_1}\circ\cdots \circ\tau_{\mathcal{D}_4}$ and $h_{\mathbb{S}^3}=\tau_{\gamma_1}\circ\cdots \circ\tau_{\gamma_4}$ are the product of right-handed Dehn twists along $\mathcal{D}_k$ and $\gamma_k$, respectively, and $A_{k}$ is the properly embedded Lagrangian $2$-disk  in the ribbon $\mathcal{R}_{(\bigsqcup_{i=1}^{k-1}M_i)\cap M_{k}}$ as described in Proposition \ref{prop:Construction_via_stabilizations} for $k=1,2,3,4$. The smoothening of $L'$ (resp. $K'$) are contained in the cores of $F_{\mathbb{S}^5}$ (resp. $F_{\mathbb{S}^3}$).\\

Now, as an example, consider the admissible relative contact pair $(Y_{\xi}^5, M_{\eta}^3)$ obtained by applying an admissible relative contact $(+1)$-surgery along the given admissible link $L$ in $\mathbb{S}_{st}^5$. By Theorem \ref{thm:contact surgery}, there exists a compatible relative open book $$(\mathcal{OB}_Y, \mathcal{OB}_M)\doteq (F_{\mathbb{S}^5},F_{\mathbb{S}^3},h_{\mathbb{S}^5}\circ\tau_{L'}^{-1},h_{\mathbb{S}^3}\circ\tau_{K'}^{-1})$$ on $(Y_{\xi}^5, M_{\eta}^3)$ . Note that $Y_{\xi}^5$ is overtwisted by Casal, Murphy and Presas \cite{CMP} since it is obtained by contact $(+1)$-surgery along a loose Legendrian $2$-knot $L'$. In particular, it contains the overtwisted contact submanifold  $M_{\eta}=(L(2,1)=\mathbb{RP}^3,\eta)$ which is obtained by contact $(+1)$-surgery along the equatorial $1$-knot $K'$.
	\end{example}



\begin{thebibliography}{99}
		
\bibitem{AO} S. Akbulut, B. Ozbagci, \emph{Lefschetz fibrations on compact Stein surfaces}, Geom. Topol. 5 (2001), 319-334.
		
\bibitem{Ar} M. F. Arikan, \emph{On the support genus of a contact structure}, Journal of Gokova Geometry Topology, Vol. 1 (2007), 92-115.
		
\bibitem{Av} R. Avdek, \emph{Liouville hypersurfaces and connect sum cobordisms}, Journal of Symplectic Geometry, Vol. 19(4), 2021.
		
\bibitem{CMP} R. Casals, E. Murphy, F. Presas, \emph{Geometric criteria for overtwistedness},Journal of American Mathematical Society, 32 (2), (2015).
		
\bibitem{CE} J. Conway, J. Etnyre, \emph{Contact surgery and symplectic caps}, Bull. Lond. Math. Soc., 52(2), 379-394 .
		
\bibitem{DG} F. Ding, H. Geiges, \emph{A Legendrian surgery presentation of contact $3$-manifolds}, Math. Proc. Cambridge Philos. Soc. 136 (2004), no. 3, 583-598.
		
\bibitem{EES} T. Ekholm, J. Etnyre, M. Sullivan, \emph{Non-isotopic Legendrian submanifolds in $\R^{2n+1}$},  J. Differential Geom. 71 (2005), no. 1, 85--128.
		
\bibitem{Et1} J. B. Etnyre, \emph{Lectures on open book decompositions and contact structures, Floer homology, gauge theory, and low-dimensional topology}, 103-141, Clay Math. Proc., 5, Amer. Math. Soc., Providence, RI, (2006).
		
\bibitem{EO} J. B. Etnyre, B. Ozbagcı, \emph{Invariants of Contact Structures from Open Books}, Trans. Amer. Math. Soc. Vol. 360 (2008), 3133-3151.
		
\bibitem{Geiges} H. Geiges, \emph{An Introduction to Contact Topology}, Cambridge University Press, (2008).
		
\bibitem{Gi} E. Giroux, \emph{G\'eom\'etrie de contact: de la dimension trois vers les dimensions sup\'erieures}, Proceedings of the International Congress of Mathematicians, Vol. II (Beijing), Higher Ed. Press,
		(2002), pp. 405--414. MR 2004c:53144
		
\bibitem{Gm} R. E. Gompf, \emph{Handlebody construction of Stein Surfaces}, Ann. of Math. 148(1998), 619-693.
		
\bibitem{Gd} N.Goodman, \emph{Contact Structures and Open Books}, PhD thesis, University of Texas at Austin(2003).
		
\bibitem{Koert} O. V. Koert, \emph{Open books on contact five manifolds}, Tome 58 (2008), p.139-157.
		
\bibitem{Koert2} O. V. Koert, \emph{Lecture notes on stabilization of
			contact open books}, M¨unster J. of Math. 10 (2017), 425–455.
		
\bibitem{L } O. Lazarev, \emph{Maximal contact and symplectic structures}, Journal of Topology 13,2020,1058-1083.
		
\bibitem{Ly} H.Lyon, \emph{Torus knots in the complements of links and surfaces}, Michigan Math. J. 27 (1980), 39-46.
		
\bibitem{MS} D. McDuff, D. Salamon, \emph{Introduction to Symplectic Topology}, Oxford University Press, (1995).
		
\bibitem{OS} B. Ozbagci, A. I. Stipsicz,  \emph{Surgery on contact $3$-manifolds and Stein surfaces}, Bolyai Society Mathematical Studies, 13 (2004), Springer-Verlag, Berlin.
		
\bibitem{Seidel} P. Seidel, \emph{Floer homology and the	symplectic isotopy problem}, Ph.D thesis, University ofOxford (1997).

\end{thebibliography}
\end{document}